\documentclass[journal,twoside,web]{ieeecolor}
\usepackage{generic}
\usepackage{amsmath,amssymb,amsfonts,bm}
\usepackage{algorithmic}
\usepackage{graphicx}
\usepackage{algorithm,algorithmic}
\usepackage{hyperref}
\usepackage[thmmarks, amsmath]{ntheorem}
\let\labelindent\relax
\usepackage{enumitem}
\usepackage{mathtools}
\usepackage{textcomp}

\usepackage[style=ieee]{biblatex}
\theoremstyle{plain} 
\theoremnumbering{arabic} 
\theoremseparator{.} 
\newtheorem{theorem}{Theorem}
\newtheorem{lemma}{Lemma}

\newtheorem{proposition}{Proposition}

\theoremstyle{plain} 
\theoremheaderfont{\bfseries\itshape} 
\theoremnumbering{arabic} 
\theoremseparator{:}
\newtheorem{scenario}{Scenario}

\theoremstyle{plain} 
\theoremheaderfont{\itshape} 
\theorembodyfont{\itshape} 
\theoremnumbering{arabic} 
\theoremseparator{:} 
\newtheorem{definition}{Definition}
\newtheorem{assumption}{Assumption}

\theoremstyle{plain}
\theoremheaderfont{\itshape}  
\theorembodyfont{\normalfont} 
\theoremseparator{:}          

\newtheorem{remark}{Remark}

\usepackage[labelformat=simple]{subcaption}

\DeclareMathOperator*{\argmin}{arg\,min}
\newcommand\numberthis{\addtocounter{equation}{1}\tag{\theequation}}
\newcommand{\alglinelabel}{%
  \addtocounter{ALC@line}{-1}
  \refstepcounter{ALC@line}
  \label
}

\usepackage{chngcntr}
\makeatletter 
\newcounter{parentassumption} 

\newenvironment{assumptionset}{%
  \refstepcounter{assumption}
  \setcounter{parentassumption}{\value{assumption}}
  \protected@edef\theparentassumption{\theassumption}
  \setcounter{assumption}{0}
  \renewcommand{\theassumption}{\theparentassumption\alph{assumption}}
}{%
  \setcounter{assumption}{\value{parentassumption}}
}
\makeatother 

\allowdisplaybreaks

\def\BibTeX{{\rm B\kern-.05em{\sc i\kern-.025em b}\kern-.08em
    T\kern-.1667em\lower.7ex\hbox{E}\kern-.125emX}}
\begin{document}
\title{System Identification under Noise and Attack Regimes: Agnostic and Composite Robustness}
\author{Jihun Kim, \IEEEmembership{Graduate Student Member, IEEE} and Javad Lavaei, \IEEEmembership{Fellow, IEEE}
\thanks{This work was supported by the U. S. Army Research Laboratory and the U. S. Army Research Office under Grant W911NF2010219, Office of Naval Research under Grant N000142412673, and NSF.}
\thanks{Jihun Kim and Javad Lavaei are with the Department of Industrial Engineering and Operations Research, University of California, Berkeley, CA 94720 USA (emails:  {\tt\footnotesize \{jihun.kim, lavaei\}@berkeley.edu}).} 
\thanks{A preliminary version of this paper has been accepted for the 65$^{\text{th}}$ IEEE Conference on Decision and Control (CDC), Honolulu, Hawaii, USA, Dec 15-18, 2026 \cite{kim2026huber}. The previous version primarily discussed systems subject to either a noise or an attack regime without prior knowledge. In this journal paper, we further investigate the composite  case in which noise and attacks are simultaneously injected into the system. We refer to accurate system identification for the former setting as agnostic robustness, and for the latter as composite robustness. While we used the one-stage Huber estimator for the agnostic setting, we leverage a two-stage approach for the composite case by sequentially applying the median-based estimator to detect and filter out attacks, followed by the mean-based estimator to yield an estimation error of the order $\mathcal{O}(1/\sqrt{T})$ plus the product of the noise level and the number of false detections.}
}
\maketitle

\begin{abstract}
Dynamical systems often confront persistent zero-mean independent noise and/or sparse nonzero-mean adversarial attacks. While mean-based estimators like least-squares handle the former, the median-based $\ell_1$-norm estimator is effective for the latter. In this paper, we develop robust system identification frameworks to identify a linearly-parametrized nonlinear system from a single trajectory of length $T$. We tackle two types of robustness: (1) \textit{agnostic robustness} under either pure noise or pure attacks without knowing which regime is active; and (2) \textit{composite robustness} under concurrent noise and attacks. We first show that the Huber estimator attains agnostic robustness by achieving an $\mathcal{O}(1/\sqrt{T})$ error rate for the noise regime and a bounded error for the attack regime, serving as a one-stage estimator interpolating between mean- and median-based methods. We then prove that no convex one-stage estimator is consistent under both noise and attacks, which motivates a two-stage estimation that sequentially applies median- and mean-based estimators for composite robustness. These dual notions of robustness require a corresponding duality in estimator design, providing a solid foundation for robust control in safety-critical systems.
\end{abstract}


\section{Introduction}
\label{sec:intro}

Control systems are generally subject to exogenous disturbances; in particular, noise and attacks. These factors propagate through system dynamics, embedding themselves in temporally correlated states. \textit{Noise} is generally a factor that does not directly corrupt the system, yet is unavoidable: it persists within the system states as a white stochastic disturbance present at every time \cite{freidlin2012random}.
 Its sources include natural fluctuations via fundamental environmental physics \cite{bentley2005principles}, unmodeled internal dynamics due to device imperfections \cite{vanderziel1978noise}, and digital processing or quantization errors \cite{widrow2008quantization}. For this reason, it is reasonable to model the process noise as  
 \textbf{\textit{persistent zero-mean independent noise}} 
  that does not intentionally corrupt the system and can even be symmetric. On the other hand, \textit{attacks} are intentional manipulations that can be severe enough to mislead states of the system: they appear in security and fault-diagnosis settings, where an adversary intermittently but strategically corrupts and/or biases the system based on the potentially full information history available at each time \cite{pasqualetti2013cyber, teixeira2015secure}; we therefore refer to them as \textbf{\textit{sparse nonzero-mean adversarial attacks}}. 

While systems can be subject to either \textit{noise, attacks, or even both}, detecting attacks is of primary interest, since they can adversely affect the system. Indeed, a system's model can be learned  only when the occurrence of attacks is at most 50\%; if the system is consistently corrupted beyond this threshold,  distinguishing  ``good'' from ``bad'' measurements becomes impossible. 
This is formalized by the theoretical limit that the number of correctable errors is at most half of all states \cite{fawzi2014secure}.
Building on this principle, attack detection schemes have been widely developed, especially for cyber-physical systems where attacks occur infrequently but can be maliciously large \cite{zhang2015optimal}. Notably, the work \cite{pajic2017attack} considered systems affected by both attacks and bounded noise, aiming to detect attacks in the presence of noise.  
However, the analyses in aforementioned works assume that the system dynamics is known a priori and focus on reconstructing the original states from  noisy or attacked measurements. In practice, having access to the system dynamics is overly optimistic and often unrealistic.

As modern systems grow increasingly complex, entirely or partially unknown system dynamics has motivated the development of \textit{system identification} methods, the procedure of learning underlying models based on the state/output trajectory generated by the system, providing the foundation to design robust and reliable control algorithms \cite{ljung1998sys}. For large-scale infrastructure such as power systems, however, collecting data via forcing state resets  often incurs massive operational downtime \cite{ghodeswar2025quantifying}. Similarly, human patients in clinical control  cannot be physiologically reset \cite{allam2021analyzing}. Thus, practical system identification in many real-world applications necessitates learning from a single trajectory. We formulate this task as a parameter estimation problem and consider a linearly-parametrized nonlinear dynamical system generating a sequence of the form
\begin{align}\label{linearlyparam}
    x_{t+1} = \bar A \phi(x_t) +  w_t,\quad t=0,\dots, T-1,
\end{align}
where $x_t\in\mathbb{R}^n$ is the state, $w_t\in\mathbb{R}^n$ is the disturbance at time $t$, and $T$ denotes the trajectory length. The system dynamics are governed by an unknown target matrix
$\bar A \in\mathbb{R}^{n\times m}$—with rows $\bar a_1^T,\dots, \bar a_n^T$—and known, potentially nonlinear basis functions  $\phi:\mathbb{R}^n\to\mathbb{R}^m$ chosen by the system analyst. Given a single trajectory $\{x_t\}_{t=0}^{T}$, our objective is to accurately estimate the unknown values in $\bar A$.

The primary challenge in estimating the underlying system dynamics $\bar A$ arises from the temporal correlation among states and the resulting convolved effects of the disturbances $\{w_t\}_{t=0}^{T-1}$. 
One of the earliest approaches was the classical least-squares method (see Chapter II, \cite{koopmans1950statistical}), which achieves an estimation error of $\mathcal{O}(1/\sqrt{T})$ as $T \to \infty$ when $\{w_t\}_{t=0}^{T-1}$ constitutes a persistent zero-mean independent noise. This idea was revisited in recent works \cite{simchowitz2018learning, sarkar2019near}, which established non-asymptotic guarantees, showing that an error of $\mathcal{O}(1/\sqrt{T})$ is indeed optimal and attained by the least-squares once $T$ exceeds a finite threshold. 
To address the case where $\{w_t\}_{t=0}^{T-1}$ consists of a sparse nonzero-mean adversarial attack, non-smooth convex estimators have been studied from a non-asymptotic perspective, such as the $\ell_2$-norm  or the $\ell_1$-norm estimator \cite{yalcin2024exact, kim2025prevailing}.  The two primary classes of estimators correspond to the following optimization problems:
\begin{align}
    &\min_{A\in\mathbb{R}^{n\times m}} \sum_{t=0}^{T-1} \|x_{t+1} - A\phi(x_t)\|_2^2. \tag{Least-squares}\label{ls} \\&\min_{A\in\mathbb{R}^{n\times m}} \sum_{t=0}^{T-1} \|x_{t+1} - A\phi(x_t)\|_1. \tag{$\ell_1$-norm estimator}  \label{l1}
\end{align}

It is worth noting that the least-squares method is a \textit{mean-based estimator} that averages out persistent noise under zero-mean  assumption, whereas the $\ell_1$-norm estimator is a \textit{median-based estimator} that filters out any attacks of an adversarial nature under zero-median assumption. Despite their individual efficacy, both approaches face a critical blind spot. 
The former approach is limited to zero-mean noise and fails against adversarial attacks, since biased disturbances do not average to zero. The latter approach, while robust to sparse attacks that occur with probability smaller than 0.5 and thus have a zero-median, cannot overcome persistent noise since it is almost never exactly zero and may fail to maintain a median close to zero. This leads to the central challenge of system identification:  \textit{in practice, the underlying nature of the disturbances is unknown in advance}, whether they are noise, attacks, or the composition of both.  

\textbf{Contribution. } In this paper, we provide a robust system identification framework in two different flavors: \textit{agnostic and composite robustness}. 
\begin{enumerate}[leftmargin=0.43cm
]
    \item Agnostic robustness refers to accurately identifying the underlying dynamics of the system subject to either noise or attacks, while remaining unaware of which type of disturbance it faces. In particular, we adopt the \textit{Huber estimator}, using the Huber loss \cite{huber1964robust} defined by a threshold constant $\mu$. While previous works in robust statistics \cite{Gannaz2007, SheOwen2011} demonstrated the efficacy of the Huber estimator for independent samples, we establish the first theoretical guarantees for the Huber estimator in system identification under temporally correlated states, corroborating its universal effectiveness across both a noise regime and an attack regime. Under persistent zero-mean independent noise, the Huber estimator recovers the optimal $\mathcal{O}(1/\sqrt{T})$ error rate when the noise has a positive probability density around zero. 
    Furthermore, the Huber estimator ensures that the estimation error is bounded by a constant $\mathcal{O}(\mu)$ error under sparse nonzero-mean adversarial attack. 
    
    \item Composite robustness refers to reliably recovering the system dynamics when both noise and attacks are concurrently injected into the system. In this case, we first show that any convex one-stage estimator such as the least-squares, the $\ell_1$-norm estimator, or the Huber estimator, by itself fails to identify the true system. In particular, 
     an adversary can always construct attacks that prevent convergence to the true system when the noise distribution is symmetric, even as the trajectory length $T\to \infty$. 
     Motivated by the failure of one-stage estimators in the presence of concurrent noise and attacks, we propose a \textit{two-stage estimation method} that first applies the $\ell_1$-norm estimator to detect and filter out potential attacks, and subsequently leverages the least-squares estimator to average out the residual noise. This filter-then-average concept aligns with trimmed-mean estimation for independent samples  in robust statistics \cite{lugosi2021robust}. We study this notion under temporally correlated states, showing that the two-stage estimator yields an estimation error bounded by  $\mathcal{O}(1/\sqrt{T})$ plus an additional term given by the product of the noise level and the number of false detections. In other words, the estimation performance improves when the separability between clean and corrupted data instances is high, i.e., attacks are sufficiently large compared to the noise scale to be reliably distinguished from clean data. Indeed, perfect separability provides a consistent estimate of the true system. 
\begin{figure}
    \centering
    \includegraphics[width=0.95\linewidth]{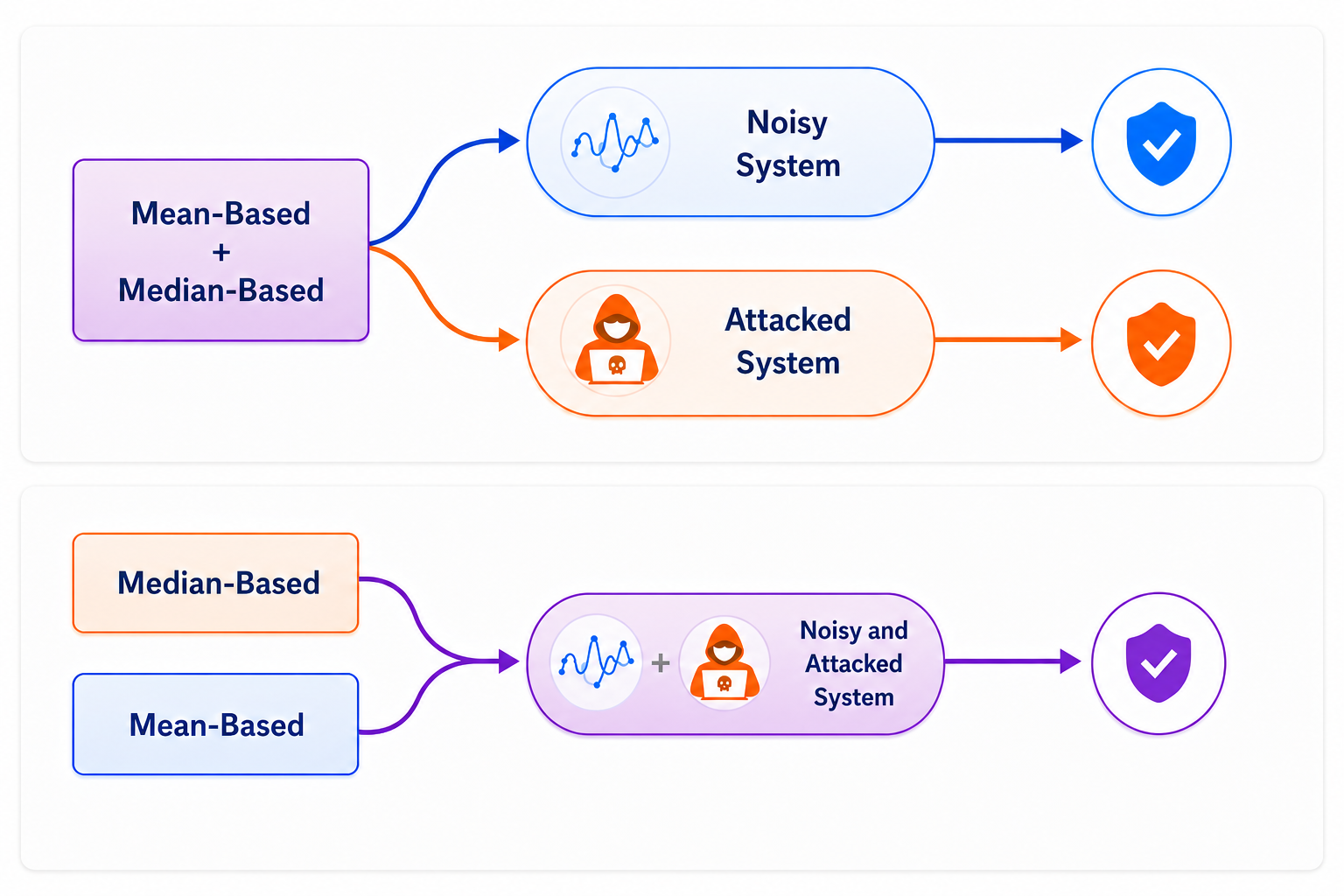}
    \caption{Agnostic (Top) and Composite (Bottom) Robustness}
    \label{fig:robustness}
\end{figure}

    \item We demonstrate that the Huber estimator used for agnostic robustness interpolates between the least-squares method and the $\ell_1$-norm estimator.
    This aggregation of mean- and median-based  approaches effectively handles either purely noisy or purely attacked systems.
    In contrast, composite robustness requires a two-stage estimation using a sequence of median- and mean-based methods to tackle concurrently noisy and attacked systems. 
    Fig.~\ref{fig:robustness} presents a diagram illustrating how these dual notions of robustness necessitate a corresponding duality in estimator design.
  \end{enumerate}
  
\textbf{Notes on Related Work. }Although \cite{kumar2025machine} recently has demonstrated the empirical robustness of  the Huber estimator when applied to neural networks,  theoretical foundations about agnostic robustness have not yet been developed. Furthermore, system identification using a single trajectory for composite robustness has received limited attention. 
While \cite{simchowitz2019semi} has explored this area by introducing known control inputs with zero mean at each time $t$, their framework models the attack as being designed based on the information history up to time $t-1$, not using the information of the current state.
Separately, \cite{kanakeri2025outlier} has used a median-of-means method to handle concurrent heavy-tailed noise and adversarial attack, but their work hinges on the existence of multiple independent trajectories and requires an attack probability much smaller than $0.5$.

 Our framework incorporates a SINDy-type structure \cite{brunton2016discovering} via sufficiently expressive nonlinear basis functions $\phi(x)$ and a sparse target matrix $\bar A$. Although the original SINDy approach relies on  least-squares suited only for zero-mean noise, our work provides theoretical guarantees to handle \textit{noise, attacks, or even both}.

\textbf{Outline. } 
In Section \ref{sec:probform}, we outline the relevant assumptions and definitions of our problem.  
Section \ref{sec:huberbased} formalizes the Huber estimator and analyze its agnostic robustness. Section~\ref{sec:onestage} demonstrates the limitations of  convex one-stage estimators for the composite case. Section~\ref{sec:twostage} proposes a two-stage estimator for composite robustness and analyzes each of its stages. 
Section \ref{sec:numexp} provides numerical validation of our claims, and Section \ref{sec:conclusion} provides concluding remarks. 

\textbf{Notation.} For a vector $x$, $x^T$  is the transpose of the vector and  $x^{(i)}$ denotes its $i^\text{th}$ entry. The notation $\|\cdot \|_2$ denotes the $\ell_2$-norm for vectors and the spectral norm for matrices, while $\|\cdot \|_1$ denotes the $\ell_1$-norm for vectors. 
$I$ denotes the identity matrix, and $\succeq$ denotes positive semidefiniteness. For a set $S$, $|S|$ denotes its cardinality. Let $\bm{\sigma}(\cdot)$ denote the sigma-algebra. $\operatorname{tr}(\cdot)$ denotes the trace operator, 
$\mathbb{E}[\cdot]$ denotes the expectation, and $\mathbb{P}(\cdot)$ denotes the probability. Let $\mathbb{I}\{\cdot\}$ denote the indicator function, taking the value $1$ if the event holds and $0$ otherwise.
A distribution $w$ is symmetric if $w$ and $-w$ are identically distributed.
Given functions $u$ and $v$, the notation $u(x) = \mathcal{O}(v(x))$ implies  $u(x) \le c_1 v(x) $ for some $c_1>0$, and $u(x) = \Omega(v(x))$ implies $u(x)\geq c_2 v(x) $ for some $c_2>0$.

\section{Problem Formulation and Core Assumptions}\label{sec:probform}

In this section, we state the assumptions and scenarios required to establish theoretical guarantees for the robustness. The first two assumptions ensure that the trajectories generated by the underlying true system do not diverge.

\begin{assumption}[System Stability]\label{as:stability}
Let $\rho$ denote the spectral norm $\|\bar A\|_2$. Let $L$ be a Lipschitz constant for $\phi:\mathbb{R}^n\to \mathbb{R}^m$; i.e., $\|\phi(x)-\phi(\tilde x )\|_2 \le L\|x-\tilde x\|_2$ for all $x,\tilde x \in\mathbb{R}^n$. Moreover, $\phi(0)= 0$.  We assume the sufficient stability condition $\rho L< 1$.
\end{assumption}

\begin{assumption}[Sub-Gaussian Disturbance]\label{as:subg}
Define $\mathcal{F}_t = \bm{\sigma}\{x_0, \dots, x_t\}.$
Assume that all $w_t$ and $x_0$ are sub-Gaussian vectors\footnote{The notion of sub-Gaussian variables is introduced in Section 2.6, \cite{vershynin2025high}. A variable $x$ is sub-Gaussian if its $\psi_2$-norm $\|x\|_{\psi_2}= \inf\{k>0: \mathbb{E}[\exp(x^2/k^2)] \le e\}$ is finite. For example, every bounded variable is sub-Gaussian.
Furthermore, a vector $X \in \mathbb{R}^n$ is defined as sub-Gaussian with $\psi_2$-norm $\sigma$ if the scalar projection $u^T X$ is sub-Gaussian with $\psi_2$-norm of at most $\sigma$ for all $u \in \mathbb{R}^n$ such that $\|u\|_2 = 1$. } (not necessarily zero-mean or independent); i.e., there exists $\sigma > 0$ such that $\|x_0^{(i)}\|_{\psi_2}\le \sigma$ and $\|w_t^{(i)}\,|\,\mathcal{F}_t\|_{\psi_2} \le \sigma$ for every $t \ge 0$ and $i\in\{1,\dots,n\}$. 
\end{assumption}

\begin{remark}\label{subgnorm}
The assumption $\phi(0) = 0$ is made for simplicity and can be relaxed for any Lipschitz $\phi$ by introducing an additive constant to the upcoming theorems.
To leverage Lipschitz continuity, it is useful to  study the $\psi_2$-norm of the $\ell_2$-norm for sub-Gaussian variables.
 Assumption \ref{as:subg} implies that $\| \|x_0\|_2 \|_{\psi_2}$ and $\| \|w_t\|_2 \|_{\psi_2}$ are bounded by $\sqrt{n}\sigma$, which is tight for independent coordinates; \textit{e.g.}, when $w_t$ follows a Gaussian  $\mathcal{N}(0, \sigma^2 I)$, its $\psi_2$-norm scales with $\sigma$ and 
 $\mathbb{E}[\|w_t\|_2^2]=\sum_{i=1}^n \mathbb{E}[(w_t^{(i)})^2] = n\sigma^2$, which means $\|w_t\|_2$ concentrates around its expected value of roughly $\sqrt{n}\sigma$. Assumption \ref{as:subg} is automatically satisfied if the initial state and disturbances are bounded.
\end{remark}

We now consider three scenarios of disturbances $\{w_t\}_{t=0}^{T-1}$: persistent zero-mean independent noise, sparse nonzero-mean adversarial attack, and the composition of both.  

\begin{scenario}[Persistent Zero-mean Independent Noise]\label{as:white}
     $w_t$ is an independent, zero-mean noise process for $t\ge0$. Moreover, $\mathbb{E}[x_0]=0$. 
\end{scenario}

\begin{scenario}[Sparse Nonzero-mean Adversarial Attack]\label{as:prob}
    $w_t$ is an attack at time $t$ with probability $p < 0.5$, conditioned on $\mathcal{F}_t= \bm{\sigma}\{x_0, \dots, x_t\}$. Formally, there exists a sequence $(\xi_t)_{t = 0}^{T-1}$ of independent $\mathrm{Bernoulli}(p)$ variables
    such that
$\{\xi_t = 0\} \subseteq \{w_t = 0\}$
for all $t\ge0$. Under this restriction on attack times, $w_t$ can be chosen arbitrarily by an adversary with access to $\mathcal{F}_t$ at every attack time $t\ge 0$.
\end{scenario}

\begin{scenario}[Composition of Noise and Attack]\label{as:composition}
    $w_t$ is decomposed as $e_t + k_t$, where $e_t$ and $k_t$ denote the noise and attack components defined in Scenarios \ref{as:white} and \ref{as:prob}, respectively. For every $t$, $e_t$ and $k_t$ are conditionally independent given $\mathcal{F}_{t-1}$. 
\end{scenario}

Scenario \ref{as:white} specifies that the system is under independent zero-mean noise at every time, while Scenario \ref{as:prob} states that the system is under adversarial attack at each time with probability smaller than $0.5$. Agnostic robustness aims to accurately identify systems regardless of whether Scenario \ref{as:white} or \ref{as:prob} applies. Composite robustness addresses Scenario \ref{as:composition} that generalizes the two aforementioned scenarios.

Persistent excitation is inherently necessary in system identification \cite{Narendra1987}. The next set of assumptions relaxes this classical notion by adopting sufficient expected excitation \cite{zhang2025exact, kim2026sharp}. Note that Scenarios \ref{as:white} and  \ref{as:composition} require Assumption \ref{as:excitation} since each disturbance contains noise, while 
Scenario \ref{as:prob} relies on 
Assumption \ref{as:excitation_attack} since it is subject to pure attacks without noise.

\begin{assumptionset}

\begin{assumption}[Expected Excitation for Scenarios \ref{as:white} and \ref{as:composition}]\label{as:excitation}
There exists $\lambda>0$ such that $\mathbb{E}[\phi(x_t)\phi(x_t)^T\;|\;  \mathcal{F}_{t-1}] \succeq \lambda^2 I$ for all $t=1,\dots, T$, meaning that $\phi(x_t)$ covers entire space in $\mathbb{R}^m$ in expectation.    
\end{assumption}

\begin{assumption}[Expected Excitation for Scenario \ref{as:prob}]\label{as:excitation_attack}
There exists $\lambda>0$ such that $\mathbb{E}[\phi(x_t)\phi(x_t)^T\;|\; \xi_{t-1}=1, \mathcal{F}_{t-1}] \succeq \lambda^2 I$ for all $t=1,\dots, T$, meaning that $\phi(x_t)$ covers entire space in $\mathbb{R}^m$ in expectation, whenever attack happens.
\end{assumption}
\end{assumptionset}

In the first half of the paper, we will establish theoretical guarantees on the estimation error of the Huber estimator across Scenarios \ref{as:white} and \ref{as:prob}. In the second half, we 
prove the necessity of the two-stage estimation method under Scenario \ref{as:composition} and present an analysis of the proposed method.

\section{Agnostic Robustness: Huber-based System Identification}\label{sec:huberbased}

\subsection{Huber Estimator}
In this section, we formalize the Huber estimator in the context of system identification and present the underlying intuition that motivates its robustness.
Given $\mu>0$, consider
\begin{align}\tag{Huber estimator}  \label{hubermin}
    \min_{\{a_i\}_{i=1}^n} \sum_{t=0}^{T-1} \sum_{i=1}^n H_{\mu} (x_{t+1}^{(i)} - a_i^T \phi(x_t)), 
\end{align}
where 
\begin{align}\label{huberfnc}
    H_{\mu} (z) = \begin{dcases*}
        \frac{1}{2}z^2 & if $|z|\le \mu$, \\
         \mu |z| -\frac{1}{2}\mu^2 & if $|z|>\mu$.
    \end{dcases*}
\end{align}
Note that the term $x_{t+1}^{(i)} - a_i^T \phi(x_t)$ is the $i^\text{th}$ entry of the residual $x_{t+1}- A\phi(x_t)$, where $a_i^T$ denotes the $i^\text{th}$ row of $A$.

The Huber loss \eqref{huberfnc} is convex and acts as a quadratic penalty for small arguments and a linear penalty for large ones. 
The following proposition formalizes how this dual behavior allows the Huber estimator to bridge between the least-squares and the $\ell_1$-norm estimators.

\begin{proposition}\label{prop1}
    The problem \eqref{hubermin} is equivalent to the problem 
    \begin{align}\label{lasso}
    \min_{A, \{v_t\}_{t\ge 0}} \sum_{t=0}^{T-1} \frac{1}{2}\|x_{t+1}  - A\phi(x_t)-v_t\|_2^2 + \mu \|v_t\|_1,
\end{align}
where $A$ is a matrix whose rows are $a_i^T$ for $i\in\{1,\dots,n\}$.
\end{proposition}

\begin{remark}\label{propremark}
 The proof is provided in our previous work \cite{kim2026huber}. While previous works \cite{Gannaz2007, SheOwen2011} established this equivalence for independent observations, we extend the formalization to temporally correlated samples $\{x_t\}_{t=0}^T$.
 The  proposition introduces an alternative convex formulation of the Huber estimator, consisting of the least-squares term  and the $\ell_1$-regularization term, implying that the Huber estimator serves as a principled middle ground between the mean- and median-based estimators. The value of $\mu$ dictates how close the Huber estimator is to either one of the two  estimators; 
 when $\mu$ is large, the heavy penalty forces $v_t$ to zero, reducing the objective nearly to  \eqref{ls}.
 Conversely, when $\mu$ is small,  $v_t$ is forced to approach the residual $x_{t+1} - A\phi(x_t)$ to minimize the quadratic loss, effectively recovering  \eqref{l1}.  Note that $\mu>0$ is strictly required to ensure that \eqref{lasso} is well-posed; at $\mu=0$, any $A$ is optimal by choosing  $v_t = x_{t+1} - A\phi(x_t)$.
\end{remark}


\subsection{Near-Optimality across Noise and Attack Regimes}\label{sec:mainresults}

This section provides theoretical guarantees for the Huber estimator across two disturbance regimes: noise and attacks. 

\subsubsection{Noise Regime}
The first main result shows that when the independent mean-zero noise is persistently injected into the system, described in Scenario \ref{as:white}, the Huber estimator achieves the optimal $\mathcal{O}(1/\sqrt{T})$ error rate. This result holds under an additional mild assumption  that the noise distribution has a positive probability density around zero. 
We formally present the assumption and the theorem.

\begin{assumption}\label{as:mildinnoise}
      There exists a universal value $q>0$ such that 
      $\mathbb{P}(|w_t^{(i)}|\le \frac{\mu}{2} ) \ge q $ holds for every $t\ge0$ and $i\in\{1,\dots,n\}$.
\end{assumption}

\begin{theorem}\label{noisescenario}
 Consider Scenario \ref{as:white} and suppose that Assumptions \ref{as:stability}, \ref{as:subg},   \ref{as:excitation}, and \ref{as:mildinnoise} hold.  Let $\hat a_1, \dots, \hat a_n$ be a minimizer to \eqref{hubermin}, and let $\bar a_1^T, \dots, \bar a_n^T$ be each row of $\bar A$. 
  Given $\delta\in(0,1)$, when 
    \begin{itemize}
       \item $w_t$ has a symmetric distribution for all $t\ge 0$, or
        \item $\phi(x) = \bar B x$ for some $\bar B\in\mathbb{R}^{m\times n}$, where $\|\bar B\|_2\le L$,
    \end{itemize}
with
 \begin{align*}
    & T = {\Omega} \biggr( \frac{n^2 L^4 \sigma^4 }{ q^2\lambda^4 (1-\rho L)^4} \log^2\left(\frac{mn}{\delta}\right)\log\left( \frac{nL\sigma }{q\lambda (1-\rho L) \delta} \right) \biggr),\numberthis\label{timeboundfinalfinal}
 \end{align*}
 it holds that
\begin{align*}
   & \|\hat a_i - \bar a_i\|_2 = \mathcal{O}\biggr(\frac{\mu \sqrt{n} L \sigma}{\sqrt{T}q\lambda^2 (1-\rho L)}\log\Bigr(\frac{n}{\delta}\Bigr)\biggr), \\&\hspace{55mm}\forall i\in\{1,\dots, n\}
\end{align*}
with probability at least $1-\delta$.
\end{theorem}

The proof of this theorem is detailed in Appendix \ref{hubernoise}. It relies on several necessary lemmas, for which the technical proofs are provided in our preliminary version \cite{kim2026huber}. First, the following lemma bounds the $\psi_2$-norm of $\|\phi(x_t)\|_2 $ for all $t\ge 0$.

\begin{lemma}\label{phixtnorm}
   Under Assumptions  \ref{as:stability} and \ref{as:subg}, we have $\bigr\| \| \phi(x_t) \|_2\bigr\|_{\psi_2}\le \frac{ \sqrt{n}L \sigma}{1-\rho L}$
   for every $t=0,\dots, T-1$.
\end{lemma}

The next lemma shows that either symmetric disturbances with nonlinear basis functions or generic zero-mean disturbances in linear systems provide tractable theoretical bounds.

\begin{lemma}\label{secondterm}
    Define 
 \begin{equation}\label{huberderivative}
       H_\mu'(z)=\begin{dcases*}
          z, & if $|z|\le \mu$\\
            \mu, & if $z> \mu$,\\
            -\mu, & if $z< -\mu$,
        \end{dcases*}
    \end{equation}
which is    the first derivative of the Huber function $H_\mu(z)$. 
  Consider Scenario \ref{as:white} and suppose that Assumptions \ref{as:stability}, \ref{as:subg},  and \ref{as:mildinnoise} hold. 
    Given $\delta\in(0,1)$, when 
    \begin{itemize}
       \item $w_t$ has a symmetric distribution for all $t\ge 0$, or
        \item $\phi(x) = \bar B x$ for some $\bar B\in\mathbb{R}^{m\times n}$, where $\|\bar B\|_2\le L$,
    \end{itemize}
    the bound
    \begin{align*}
      & \left\| \sum_{t=0}^{T-1} H_\mu'(w_t^{(i)}) \phi (x_t)\right\|_2= \mathcal{O}\biggr(\frac{\sqrt{T} \mu \sqrt{n}L\sigma}{1-\rho L}\log\Bigr(\frac{1}{\delta}\Bigr)\biggr)
    \end{align*}
holds with probability at least $1-\delta$.    
\end{lemma}

The next lemma verifies that sufficient expected excitation implies sufficient empirical excitation with high probability. 
\begin{lemma}\label{xxlowerbound}
    Suppose that Assumptions \ref{as:stability}, \ref{as:subg}, and  \ref{as:excitation} hold.  Consider any subset of $\{1,\dots, T\}$ denoted as $\mathcal{T}$ with its cardinality shown as $|\mathcal{T}|$. Given $\delta\in(0,1)$, when 
\begin{align}\label{finaltimebound}
   |\mathcal{T}| = \Omega\left( \frac{ (\sqrt{n}L \sigma)^4}{\lambda^4 (1 - \rho L)^2}  \log\Bigr(\frac{m}{\delta}\Bigr)\log\Bigr(\frac{1}{\delta}\Bigr)\right),
\end{align}
we have $\sum_{t\in \mathcal{T}} \phi(x_t)\phi(x_t)^T \succeq \frac{\lambda^2 I}{2}|\mathcal{T}|$ with probability at least $1-\delta$.
\end{lemma}

\begin{remark}
 Theorem \ref{noisescenario} demonstrates that there exists a finite time complexity beyond which the estimation error is bounded by $\mathcal{O}(\sqrt{n}\log n/\sqrt{T})$. 
We now discuss the conditions on the theorem beyond standard independent zero-mean noise (Scenario \ref{as:white}); we require (a)  a  positive probability density around zero (Assumption \ref{as:mildinnoise}), and (b) the symmetric disturbance unless the system is linear. 
In engineering practice, process noise naturally aligns with this condition: digital quantization errors are often modeled as zero-centered uniform distribution \cite{widrow2008quantization}, while thermal noise is driven by the aggregation of countless independent electron movements, which converges to a zero-mean Gaussian via the Central Limit Theorem \cite{vasilescu2005electronic}. Crucially, the noise in both of these standard scenarios is symmetric.
\end{remark}

\begin{remark}
  From an analytical perspective, Assumption \ref{as:mildinnoise} is required since the Huber estimator achieves sufficient empirical excitation exclusively via samples with small estimated errors $x_{t+1}^{(i)}-a_i^T\phi(x_t)$; the Huber loss \eqref{huberfnc} only preserves the excitation for a quadratic penalty, whereas least-squares method relies on sufficient empirical excitation averaged over all time steps. Furthermore, the symmetry of disturbances is required since applying the Huber penalty effectively truncates the estimated sample error  at $[-\mu, \mu]$, which introduces a bias if the underlying zero-mean disturbance is asymmetric.  This symmetry requirement is completely circumvented in linear systems; since linear basis functions maintain the zero-mean nature of the states $x_t$, we can derive the optimal $\mathcal{O}(1/\sqrt{T})$ error rate regardless of disturbance symmetry.
\end{remark}

\subsubsection{Attack Regime}
The second main result establishes that the estimation error of the Huber estimator is bounded by a constant $\mathcal{O}(\mu)$ under Scenario \ref{as:prob}, an adversarial attack scenario.

\begin{theorem}\label{attackscenario}
     Consider Scenario \ref{as:prob} and suppose that Assumptions \ref{as:stability}, \ref{as:subg}, and \ref{as:excitation_attack} hold.  Let $\hat a_1, \dots, \hat a_n$ be a minimizer to \eqref{hubermin}, and let $\bar a_1^T, \dots, \bar a_n^T$ be each row of $\bar A$. Given $\delta\in(0,1)$, when 
\begin{align}\label{timeboundattack}
      &  T=\Omega\biggr(\frac{(\sqrt{n}L\sigma)^4}{\lambda^4p(1-2p) } \\\nonumber&\hspace{5mm} \times \max\biggr\{   \frac{(\sqrt{n}L\sigma)^{6}\log^2(\frac{n}{\delta})}{(1-2p)\lambda^{6}(1-\rho L)^2 } , m\log\Bigr( \frac{nL\sigma }{\lambda(1-\rho L)}\Bigr) \biggr\}  \biggr),
    \end{align}
it holds that
\begin{align*}
        \|\hat a_i - \bar a_i\|_2 = \mathcal{O}\Bigr(\frac{\mu n^2 L^4\sigma^4 }{p(1-2p)\lambda^5} \Bigr), \quad \forall i\in\{1,\dots, n\}
    \end{align*}
with probability at least $1-\delta$.
\end{theorem}

  The proof is given in Appendix \ref{huberattack}. Note that the following lemma is crucial, showing that  the $\ell_1$-norm estimator perfectly recovers the system with high probability. The proof of the lemma is presented in our preliminary work \cite{kim2026huber}. 

\begin{lemma}\label{lem:excitation}
 Consider Scenario \ref{as:prob} and suppose that  Assumptions \ref{as:stability}, \ref{as:subg},   and \ref{as:excitation_attack} hold. 
    Let $f^{(i)}(a_i):= \sum_{t=0}^{T-1} |x_{t+1}^{(i)} -  a_i^T \phi(x_t)| $ for every $i$. 
    Given $\delta\in(0,1)$, when $T$ satisfies \eqref{timeboundattack}, 
there exists a constant $c>0$ such that
\begin{align*}
    f^{(i)}(a_i) - f^{(i)}(\bar a_i) \ge c T\cdot &\frac{p(1-2p)\lambda^5 }{n^2 L^4\sigma^4 } \|a_i - \bar a_i\|_2, \\&\forall a_i \in \mathbb{R}^m, \quad \forall i\in\{1,\dots, n\}
\end{align*}
with probability at least $1-\delta$.
As a by-product, 
$\bar A$ is the unique global solution to \eqref{l1} with probability at least $1-\delta$.
\end{lemma}

\begin{remark}
 Theorem \ref{attackscenario} demonstrates that after a finite time complexity, the Huber estimator shows a bounded error $\mathcal{O}(\mu)$. The proof uses Lemma \ref{lem:excitation} that the $\ell_1$-norm estimator indeed achieves a zero error within  finite time  under Scenario \ref{as:prob}; subsequently, we use the boundedness of the difference between the $\ell_1$-norm loss and the Huber loss. 
\end{remark}

Theorems \ref{noisescenario} and \ref{attackscenario} elucidate the trade-offs involved in selecting $\mu$ across two extreme disturbance regimes. For sparse nonzero-mean attacks, minimizing $\mu$ tightly bounds the estimation error, which is natural since small $\mu$ corresponds to the $\ell_1$-norm estimator as noted in Remark \ref{propremark}. 
Conversely, under persistent zero-mean noise, $\mu$ must exceed a strictly positive threshold; Assumption \ref{as:mildinnoise} requires a positive probability density across $[-\frac{\mu}{2}, \frac{\mu}{2}]$, which is a condition impossible to satisfy with $\mu=0$ for any absolutely continuous distribution. 
Importantly, however, we do not necessarily require $\mu$ to be arbitrarily large (which corresponds to the least-squares as noted in Remark \ref{propremark}), but merely needs to satisfy the assumption to achieve the optimal $\mathcal{O}(1/\sqrt{T})$ rate. Since increasing $\mu$ beyond a certain threshold does not provide benefit for persistent zero-mean noise,  our theoretical results suggest a clear tuning strategy: $\mu$ should be set to the minimal value that satisfies Assumption \ref{as:mildinnoise} for the Huber estimator to provide the optimal defense against both extreme disturbance scenarios.

\section{Failure of One-Stage Estimation in the Composite Regime}\label{sec:onestage}

In the next two sections, we consider Scenario \ref{as:composition}: a composite disturbance of both noise and attacks. In particular, this section considers a general convex optimization method applied to system identification using a single trajectory  $\{x_t\}_{t=0}^{T}$ from \eqref{linearlyparam}, where a convex loss function is applied to the prediction residual $x_{t+1} - A\phi(x_t)$. The examples include standard \eqref{ls} and \eqref{l1}, as well as \eqref{hubermin} analyzed in the previous section. 
Given that $\|\cdot \|_2^2$, $\|\cdot \|_1$,  the function \eqref{huberfnc}, and many others are convex, coercive, and even, we consider a general standard convex optimization framework used in system identification. 

\begin{definition}[Convex Optimization for System Identification]\label{def:convex}
Given a single trajectory $\{x_t\}_{t=0}^{T}$, consider a function $g: \mathbb{R}^n \to \mathbb{R}$  such that 
\begin{enumerate}
    \item $g$ is convex.    
    \item $g$ is coercive; i.e., $\lim_{\|x\|_2\to \infty} g(x) = \infty$.
    \item $g$ is even; i.e., $g(x) = g(-x)$. 
\end{enumerate}
Then, a standard framework for system identification is to minimize the aggregate loss of $g(x_{t+1}-A\phi(x_t))$ , or equivalently, 
\begin{equation} \label{convexopt}
    \min_A ~\frac{1}{T}\sum_{t=0}^{T-1} g(x_{t+1}- A\phi(x_t)) .
\end{equation}
\end{definition}

The optimization problem \eqref{convexopt} is natural in the sense that any estimator should try to minimize the distance between $x_{t+1}$ and $A\phi(x_t)$, with the belief that $A\phi(x_t)$ is the best representation of $x_{t+1}$ in average. Also, $g$ is often selected to be convex for tractability; to be coercive to penalize a large difference between $x_{t+1}$ and $A\phi(x_t)$ heavily; to be even to ensure fairness between the positive and negative differences. 
 We now demonstrate that this general framework acts as a one-stage estimator and fails to recover the true system under the composite  effect of  persistent noise and adversarial attacks; in particular, the adversary can design attacks to ensure that the estimator does not converge to the true system dynamics as long as the noise distribution is symmetric.
We now present the main theorem for the failure of one-stage estimators under the composite regime $w_t = e_t + k_t$ (see Scenario \ref{as:composition}).

\begin{theorem}\label{thm:onestage}
In Scenario \ref{as:composition}, suppose that $\{w_t\}_{t\ge 0}$ satisfies Assumption \ref{as:subg}, \ref{as:excitation} and each $e_t$ follows a symmetric and absolutely continuous distribution for every $t$. Let $A_T$ be a minimizer to \eqref{convexopt}. 
   Then, 
   given an attack probability $p>0$, 
   there exists a sequence of attacks $\{k_t\}_{t\ge 0}$ such that $A_T$ does not converge to $\bar A$  almost surely.
\end{theorem}
\begin{proof}
    Given the dynamics \eqref{linearlyparam} under Scenario \ref{as:composition}, the estimator \eqref{convexopt} can be rewritten as 
     \begin{align}\label{opt2}
         \min_A~ \frac{1}{T} \sum_{t=0}^{T-1} g((\bar A-A)\phi(x_t) + e_t+ k_t).
     \end{align}
  Note that the subdifferential of the objective function in \eqref{opt2} is $\frac{1}{T}\sum_{t=0}^{T-1} -\partial g((\bar A-A)\phi(x_t) + e_t+ k_t) \cdot \phi(x_t)^T$.

 To prove by contradiction, assume that the estimator converges to the true system $\bar A$; i.e., $A_T \xrightarrow{a.s.} \bar A$ as $T\to \infty$. 
 By the first-order optimality conditions, it holds that 
    \begin{align}\label{contradiction}
       \lim_{T\to \infty} \frac{1}{T} \sum_{t=0}^{T-1} \nabla g( e_t+ k_t) \cdot \phi(x_t)^T \xrightarrow{a.s.}0,
    \end{align}
     where the subdifferential $\partial g(e_t+k_t)$ is essentially $\{\nabla g(e_t+k_t)\}$ almost surely, since a convex $g$ is differentiable almost everywhere and $e_t$ is absolutely continuous, while independent of $k_t$. 

Since $g$ is even, its gradient $\nabla g(\cdot)$ is an odd function. Then, the term $\nabla g( e_t) \cdot \phi(x_t)^T $
    is symmetrically distributed, which follows from the symmetry of $e_t$, the independence of $e_t$ from $x_t$, and the oddness of $\nabla g$. Since the quantity
    $\frac{1}{T}\sum_{t=0}^{T-1}\operatorname{tr}( \nabla g( e_t) \cdot \phi(x_t)^T )$
    is symmetrically distributed, 
    there are two possible scenarios as $T\to\infty$:
    either it oscillates (boundedly or unboundedly), or it converges to zero almost surely.

\textbf{Case 1}: Oscillation—If $\frac{1}{T}\sum_{t=0}^{T-1}\operatorname{tr}( \nabla g( e_t) \cdot \phi(x_t)^T )$ oscillates, the adversary can simply choose $k_t=0$ at the attack times. In this case, the condition \eqref{contradiction} fails to hold. 

\textbf{Case 2}: Convergence—Let $\mathcal{A}$ and $\mathcal{N}$ denote the sets $\{0\le t\le T-1: \xi_t= 1\}$ (attack times) and $\{0\le t\le T-1: \xi_t= 0\}$ (non-attack times), respectively.
If $\frac{1}{T}\sum_{t=0}^{T-1}\operatorname{tr}( \nabla g( e_t) \cdot \phi(x_t)^T )$ converges to zero almost surely,  then by Lemma \ref{borelcantelli}, the subsampled average at the non-attack times ($k_t = 0$) satisfies
\begin{align}\label{nonattackconverge}
    \frac{1}{T}\sum_{t\in\mathcal{N}}\operatorname{tr}( \nabla g( e_t) \cdot \phi(x_t)^T ) \xrightarrow{a.s.}0.
\end{align}

For the analysis of the attack times, we first let $\mathcal{G}_{t} = \bm{\sigma}(\mathcal{F}_{t-1} \cup \{k_t\})$ denote the augmented filtration that includes the realization of the adversary's action, and observe that 
 $e_t$ is symmetrically distributed conditional on $\mathcal{G}_{t}$. Then,  $\operatorname{tr} (\nabla g( e_t+ k_t) \cdot \phi(x_t)^T )$ and $\operatorname{tr} (\nabla g( -e_t+ k_t) \cdot \phi(x_t)^T )$ have the same   conditional distributions given $\mathcal{G}_{t}$. Consequently, the sequence of their differences forms a conditionally symmetric martingale difference sequence with respect to $\mathcal{G}_{t}$. Thus, the empirical average of the halved differences over $\mathcal{A}$ given by 
 \begin{align}\label{diff0}\frac{1}{2T} \sum_{t\in\mathcal{A}} \operatorname{tr} ((\nabla g( e_t+ k_t) - \nabla g( -e_t+ k_t) )\cdot \phi(x_t)^T ) ,\end{align}
 must satisfy $\limsup_{T \to \infty} (\cdot) = -\liminf_{T \to \infty} (\cdot)$ almost surely. This symmetry leaves only two possibilities as $T\to \infty$: the quantity either oscillates (boundedly or unboundedly), or  converges to zero almost surely. 
 In both cases, it suffices to prove that the limit
 \begin{align}\label{nondiff0}\liminf_{T\to \infty}\frac{1}{2T} \sum_{t\in\mathcal{A}} \operatorname{tr} ((\nabla g( e_t+ k_t) + \nabla g( -e_t+ k_t) )\cdot \phi(x_t)^T ) ,\end{align}
is lower-bounded by a positive constant to contradict \eqref{contradiction}, since the expression in \eqref{contradiction} decomposes exactly into the sum of the terms in \eqref{nonattackconverge}, \eqref{diff0}, and \eqref{nondiff0}.

To analyze \eqref{nondiff0}, we observe  that $\nabla g(\cdot)$ is odd and monotonic. It follows that
{\small \begin{align*}
      &\nabla g( e_t+ k_t)^T k_t +  \nabla g( -e_t+ k_t)^T k_t \\&=\nabla g( e_t+ k_t)^T k_t -  \nabla g( e_t- k_t)^T k_t\numberthis\label{genericwt}\\&=  \frac{1}{2}(  \nabla g( e_t+ k_t)  -  \nabla g( e_t- k_t) )^T ((e_t+k_t) - (e_t - k_t))  \ge 0
    \end{align*}}
for all $e_t, k_t \in \mathbb{R}^n$. 

By Assumption \ref{as:subg},  $w_t$ has a universal $\psi_2$-norm $\sigma$ for all $t$, which thus holds for $e_t$. This implies that there exists $K>0$ such that 
    $\mathbb{P}(\|e_t - \mathbb{E}[e_t]\|_2 \le K\sqrt{n} \sigma) \ge \frac{1}{2}.$
Consider 
\begin{align*}
    \mathcal{\bar W}_t = \{e_t\in \mathbb{R}^n : \|e_t - \mathbb{E}[e_t]\|_2 \le K\sqrt{n} \sigma\},
\end{align*}
which implies $\mathbb{P}(e_t \in \mathcal{\bar W}_t)\ge \frac{1}{2}$.
From the compactness of $\mathcal{\bar W}_t$ and the continuity of $\|\cdot\|_2$ and $g$, there exists $M_1, M_2>0$ such that 
\begin{align}\label{maxcompact}
  \|e_t\|_2 \le M_1 ~~\text{and}   ~~g(e_t)\le M_2, \quad \forall e_t\in\mathcal{\bar W}_t.
\end{align}
By the coercivity of $g$, there exists $R\ge 0$ such that 
\begin{align}\label{inftydef}
    \|x\|_2\ge R \quad \Longrightarrow \quad g(x) \ge 2M_2. 
\end{align}

At the attack times, when $\phi(x_t)\ne 0$, let the attack design be $k_t = (R+M_1) \frac{\phi(x_t)}{\|\phi(x_t)\|_2}$. By \eqref{maxcompact}, for all $e_t \in \mathcal{\bar W}_t$, we have $\|e_t + k_t\|_2 \ge \|k_t \|_2 - \|e_t\|_2 \ge (R+M_1) - M_1 = R$. Under \eqref{inftydef}, this ensures $g(e_t + k_t)  \ge 2M_2$, and thus from \eqref{maxcompact}, $g(e_t + k_t) -g(e_t) \ge M_2$. From the convexity of $g$, it follows that 
\begin{align*}
   \nabla g( e_t+ k_t)^T k_t 
    \ge g(e_t+k_t) - g(e_t)  \ge M_2.
\end{align*} Similarly, one can also obtain
    $\nabla g( -e_t+ k_t)^T k_t 
    \ge  M_2,$
    which sums up to 
    \begin{align}\label{strictwt}
       \nabla g( e_t+ k_t)^T k_t +  \nabla g( -e_t+ k_t)^T k_t\ge 2M_2 >0
    \end{align}
for all $e_t \in \mathcal{\bar W}_t$.
Combining \eqref{genericwt} and \eqref{strictwt}, we attain
\begin{align*}
    \nabla g( e_t+ k_t)^T k_t +  \nabla g( -e_t+ k_t)^T k_t\ge 2M_2 \cdot \mathbb{I} \{e_t \in \mathcal{\bar W}_t\}.
\end{align*}
From the design of $k_t$, we can write $\phi(x_t) = k_t \frac{\|\phi(x_t)\|_2 }{R+M_1 }$. Then, it follows that 
\begin{align*}
  & \frac{1}{2} \operatorname{tr} ((\nabla g( e_t+ k_t)+\nabla g(-e_t+ k_t)) \cdot \phi(x_t)^T )\\& =\frac{1}{2} [ \nabla g( e_t+ k_t)^T k_t +  \nabla g( -e_t+ k_t)^T k_t ]\frac{\|\phi(x_t)\|_2 }{R+M_1 }\\&\hspace{23mm}\ge M_2  \cdot \mathbb{I} \{e_t \in \mathcal{\bar W}_t\}\frac{\|\phi(x_t)\|_2 }{R+M_1 } := Y_t.\numberthis\label{yt}
\end{align*}
Since $\|\phi(x_t)\|_2$ is sub-Gaussian and $\mathbb{I} \{e_t \in \mathcal{\bar W}_t\}$ is bounded by $[0,1]$, $Y_t$ is a sub-Gaussian (obviously with bounded variance). Thus, $\{\sum_{t=0}^{k} (Y_t - \mathbb{E}[Y_t\;|\;\mathcal{F}_{t-1}])\}$ is a Martingale with respect to the filtration $\{\mathcal{F}_k\}$. We invoke the Strong Law of Large Numbers (SLLN) for Martingales \cite[Theorem 2.18]{hall1980martingale} to obtain 
\begin{align}\label{martingale}
    \lim_{T\to \infty}\frac{1}{T} \sum_{t\in \mathcal{A}} (Y_t - \mathbb{E}[Y_t\;|\;\mathcal{F}_{t-1}])\xrightarrow{a.s.} 0. 
\end{align} Moreover, we have
\begin{align*}
    \mathbb{E}[Y_t \;|\;\mathcal{F}_{t-1}] &= \frac{M_2}{R+M_1} \mathbb{E}[\mathbb{I} \{e_t \in \mathcal{\bar W}_t\}\|\phi(x_t)\|_2 ~|~\mathcal{F}_{t-1} ]\\&\ge \frac{M_2}{2(R+M_1)} \mathbb{E}[\|\phi(x_t)\|_2 ~|~\mathcal{F}_{t-1} ]\\&\ge \frac{M_2}{2(R+M_1)}\cdot \frac{\mathbb{E}[\|\phi(x_t)\|_2^2 \;|\;\mathcal{F}_{t-1} ]^2}{\mathbb{E}[\|\phi(x_t)\|_2 ^3\;|\;\mathcal{F}_{t-1} ]},
\end{align*}
where the first inequality is from   $\mathbb{P}(e_t \in \mathcal{\bar W}_t)\ge \frac{1}{2}$ and $e_t$ being independent of $x_t$ and the second inequality invokes the Cauchy-Schwarz inequality. 
From Assumption \ref{as:excitation}, $\mathbb{E}[\|\phi(x_t)\|_2^2\;|\;\mathcal{F}_{t-1}]  = \operatorname{tr}(\mathbb{E}[\phi(x_t)\phi( x_t)^T\;|\;\mathcal{F}_{t-1}]) \ge \lambda^2 n $. Together with  Jensen's inequality for a convex function $1/x$, we have 
\begin{align*}
  &\liminf_{T\to \infty}  \frac{1}{T}\sum_{t\in\mathcal{A}}  \mathbb{E}[Y_t \;|\;\mathcal{F}_{t-1}]  \\&\ge \frac{M_2 \lambda^4 n^2}{2(R+M_1)}\cdot \liminf_{T\to \infty}\frac{1}{T}\sum_{t\in\mathcal{A}} \frac{1}{\mathbb{E}[\|\phi(x_t)\|_2 ^3\;|\;\mathcal{F}_{t-1} ]} \\&\ge\frac{M_2 \lambda^4 n^2}{2(R+M_1)}\cdot \biggr(  \limsup_{T\to \infty}\frac{1}{T}\sum_{t\in\mathcal{A}} \mathbb{E}[\|\phi(x_t)\|_2 ^3\;|\;\mathcal{F}_{t-1} ]\biggr)^{-1}.
\end{align*}
Combining this with \eqref{martingale} and using the established upper bound $C$ based on Lemma \ref{EphiFbounded}, one can conclude that 
\begin{align}\label{ytlb}
    \liminf_{T\to \infty} \frac{1}{T} \sum_{t\in \mathcal{A}} Y_t \ge \frac{M_2 \lambda^4 n^2 }{2(R+M_1)C}.
 \end{align}
Considering \eqref{yt}, the expression in \eqref{nondiff0} is lower-bounded by $\frac{M_2 \lambda^4 n^2}{2(R+M_1)C}>0$. This completes the proof.
\end{proof}

For the proof of the failure of one-stage estimators, we have leveraged two necessary lemmas, for which the  proofs are provided in Appendix \ref{lemmasproof4}. First, the following lemma shows that the subsampling preserves almost sure convergence to zero. 
\begin{lemma}\label{borelcantelli}
    Let $\{e_t\}_{t\ge 0}$ be a sequence of independent symmetric random variables in $\mathbb{R}^n$. Let $h_t: \mathbb{R}^n \to \mathbb{R}$ be a measurable odd function for every $t$. Consider the quantity $Z_T = \frac{1}{T} \sum_{t=0}^{T-1} h_t(e_t)$ and let $(\xi_t)_{t=0}^{T-1}$ be a sequence of $\mathrm{Bernoulli}(p)$ with $p\in(0,1]$, independent of $e_t$. Define the subsampled average $\tilde Z_T = \frac{1}{T}\sum_{t=0}^{T-1} (1-\xi_t)h_t(e_t)$. If $Z_T$ converges to zero almost surely as $T\to \infty$, so does $\tilde Z_T$. 
\end{lemma}

The next lemma establishes that the long-term average of the conditional moment of $\phi(x_t)$ remains bounded.
\begin{lemma}\label{EphiFbounded}
 Under Assumptions \ref{as:stability} and \ref{as:subg}, there exists a constant $C>0$ such that $\limsup_{T\to \infty} \frac{1}{T}\sum_{t=0}^{T-1} \mathbb{E}[\|\phi(x_t)\|_2 ^3\;|\;\mathcal{F}_{t-1} ]\le C$.     
\end{lemma}

\begin{remark}
Although convex one-stage estimators can resolve isolated regimes (e.g., $\mathcal{O}(1/\sqrt{T})$ convergence rate in the noise regime via \eqref{ls}, and exact recovery in the attack regime via \eqref{l1}),
Theorem \ref{thm:onestage} establishes the fundamental non-convergence toward the true system for the unconstrained estimation of $A$ under the \textit{composite regime}. In particular, the attack is designed to have a sufficiently large magnitude to mislead any one-stage estimator; note the design $k_t = (R+M_1) \frac{\phi(x_t)}{\|\phi(x_t)\|_2}$, where $R$ is a sufficiently large constant based on the coercivity of the function $g$. This necessitates a two-stage estimation framework for the composite robustness, especially for detecting and filtering out large attacks. 
\end{remark}

\begin{remark}
     Moreover, adding norm constraints or regularization on $A$ (such as in SINDy \cite{brunton2016discovering}) does not resolve this theoretical limit. If the constraint bound on $A$ is tight, it explicitly eliminates the true system $\bar{A}$ from the feasible set; if it is loose, the problem simply reduces back to the unconstrained estimation scenario. Hence, such constraints serve only to regularize the estimate and inevitably introduce a constant bias, meaning the fundamental impossibility of asymptotic convergence to the true system remains unchanged.
\end{remark}

\section{Composite Robustness: Two-Stage Estimation Method}\label{sec:twostage}

In this section, we propose a two-stage estimation method whose objective is to detect large attacks and discard the corresponding time instances to obtain clean data. Once clean data are obtained, we leverage the classical least-squares to produce an accurate estimate of the true system. 
Our theoretical results hold for any noise and attack magnitudes, but the derived estimation bound is significantly tighter when the attacks are large compared to the noise level. In other words, our method can still be challenged by attacks whose magnitude is comparable to or smaller than the noise. 
However, an attacker typically aims to maximize the damage to the system, which is often done with large attacks by making the states deviate from their nominal values. Since small attacks lack the power to severely impact the system and can safely be treated as noise, our focus on large attacks plus small noise is reasonable in many practical scenarios.
In such scenarios, our method detects and removes sufficiently large attacks using the proposed filtering procedure.
We note that this regime cannot be addressed within a one-stage framework, as established in Theorem~\ref{thm:onestage}. This demonstrates that our two-stage estimation framework outperforms one-stage estimators. We present our framework in Algorithm \ref{algo1}.

\begin{algorithm}[t]
    \caption{Two-Stage Estimation using the $\ell_1$-norm estimator and the least-squares}
    \begin{flushleft}
        \textbf{Input: } A trajectory of length $T: (x_0, \dots, x_{T})$. Detection threshold parameters $\beta_1, \beta_2> 0$.
    \end{flushleft}
    {
\setlength{\belowdisplayskip}{1.2pt}
    \begin{algorithmic}[1]
            \STATE (\textbf{Stage I}) Solve $n$ optimization problems
            \begin{align}\label{stage1}
                \min_{a_i} f_T^{(i)}(a_i) := \sum_{t=0}^{T-1} |x_{t+1}^{(i)}-a_i^T\phi(x_t)|
            \end{align}
            and let $\mathring a_i \in \argmin_{a_i} f_T^{(i)}(a_i)$ for all $i=1,\dots, n$.
            
            \STATE (\textbf{Filtering}) For each $i$,  collect time indices  corresponding to potential clean data  as 
            \begin{equation}\label{filter}
             \begin{split}
                 \mathcal{T}_i = \{t\in& \{0,\dots, T-1\}: \\&|x_{t+1}^{(i)} - \mathring a_i^T \phi(x_t)|\le \beta_1 \|\phi(x_t)\|_2+ \beta_2\}.
                 \end{split}
            \end{equation}\alglinelabel{line2}
            \STATE (\textbf{Stage II}) Solve $n$ least-squares optimization problems to find the point estimate
        \begin{align}\label{least}
                \min_{a_i} L_T^{(i)}(a_i) := \sum_{t\in \mathcal{T}_i} (x_{t+1}^{(i)}-a_i^T\phi(x_t))^2
            \end{align} and let $\tilde a _i \in \argmin_{a_i} L_T^{(i)}(a_i)$ for all $i=1,\dots,n.$
    \end{algorithmic}
    }
    \begin{flushleft}
        \textbf{Output: } 
       Stage I : $\mathring A = \begin{bmatrix}
\mathring a_i^T
\end{bmatrix}_{i=1}^n$, Stage II : $\tilde A = \begin{bmatrix}
\tilde a _i^T
\end{bmatrix}_{i=1}^n$, with each vector stacked as rows.
    \end{flushleft}
    \label{algo1}
\end{algorithm}

\subsection{Algorithm Description}\label{subsec:algodescription}
The algorithm consists of two stages, \textbf{Stage I} and \textbf{Stage II}. In Stage I, our goal is to detect large attacks. To this end, we seek an estimate of the true system $\bar A=\begin{bmatrix}
\bar a_i^T
\end{bmatrix}_{i=1}^n$ that is sufficiently accurate in the presence of both noise and attacks. Since no one-stage estimator can produce an arbitrarily accurate estimate of $\bar A$, we can at best obtain an estimate with a small constant error. This estimate can then be leveraged to classify data along the trajectory into ``clean'' and ``corrupted'' instances, which is feasible when $\bar A$ is estimated with adequate accuracy. For this stage, we adopt a row-wise $\ell_1$-norm estimator, as given in~\eqref{stage1}, which produces
$\mathring A = \begin{bmatrix}
\mathring a_i^T
\end{bmatrix}_{i=1}^n$
 with index $i$ corresponding to each node. 

Before proceeding to Stage II, we use the estimate $\mathring a_i^T$ for each node $i$ to filter out data points suspected of being attacked. Specifically, we compare the residual
$|x_{t+1}^{(i)}-\mathring a_i^T \phi(x_t)|=|(\bar a_i - \mathring a_i)^T \phi(x_t) + e_t^{(i)} + k_t^{(i)}|$
against the threshold $\beta_1 \|\phi(x_t)\| + \beta_2$, where $\beta_1,\beta_2$ are tuning parameters. 
Given that $\mathring a_i$ is sufficiently close to $\bar a_i$, 
 we expect the residual to be small when $k_t^{(i)} = 0$ and detectably large enough when $k_t^{(i)} \neq 0$. Based on this criterion, we collect clean data for each node $i$ from the original time indices $\{0,\dots,T-1\}$ by retaining those instances whose residuals fall below the threshold.

In Stage II, we obtain an estimate of the true dynamics $\bar A$ using the collection of (potentially) clean data for each node. Assuming that this dataset is perfectly clean, i.e., contains only noise and no attacks, the least-squares estimator guarantees the optimal estimation error bound of $\mathcal{O}(1/\sqrt{T})$.

\begin{remark}\label{coordinate}
  Note that we adopt row-wise $\ell_1$-norm estimators instead of the full-matrix $\ell_1$-norm estimator. 
  It decomposes $\|\bar A - \mathring A\|$ into $n$ separate estimation errors $\|\bar a_i - \mathring a_i\|$, allowing us to independently address different adversarial strategies across coordinates. To be specific, we can relax the temporal sparsity assumptions in the attack model of Scenario \ref{as:prob}, requiring the attack to be either nonzero or entirely zero, to ``coordinate-wise attacks''; i.e., for each coordinate $i$ and time $t$, there exists an independent $\mathrm{Bernoulli}(p)$ ($p<0.5$) sequence $(\xi_t^{(i)})_{t = 0}^{T-1}$ satisfying $\{\xi_t^{(i)} = 0\} \subseteq \{k_t^{(i)} = 0\}$, 
  under which the adversary retains the ability to arbitrarily select $k_t$ using $\mathcal{F}_t$.
  In a networked system, this relaxed assumption means that at least one node is attacked at each time with probability $1-(1-p)^n\approx 1$; effectively, the overall network is under almost-persistent local attack. 
  This relaxation is viable in the composite regime since the noise $e_t$ can provide a sufficient excitation (Assumption \ref{as:excitation}) across all coordinates. In the absence of noise in Scenario \ref{as:prob}, however, even a single coordinate not attacked receives zero external input and prevents the full matrix $\phi(x_t)\phi(x_t)^T$ from being excited; thus, Assumption \ref{as:excitation_attack} almost always breaks down under  coordinate-wise attacks.
    \end{remark}

\begin{remark}\label{remark:false}
In Stage II, we apply least-squares to the filtered dataset, assuming all data are clean, i.e., $e_t^{(i)} + k_t^{(i)} = e_t^{(i)}$. In practice, the filtering procedure inevitably produces misclassifications unless attacked and non-attacked data are perfectly separable. False negatives, which are attacked data that remain in the dataset $\mathcal{T}_i$, occur when $|e_t^{(i)} + k_t^{(i)}|$ is not large (or bounded by some $M>0$), introducing a bias of at most $M$. False positives, which are clean data removed from the dataset due to unusually large $|e_t^{(i)}|$, also induce bias by truncating a tail of the noise distribution. Both types of errors directly or indirectly bias the filtered dataset, degrading the performance of the least-squares estimator.
\end{remark}


\subsection{Analysis of Two-Stage Estimation} \label{sec:analysis}

In this section, we present the main theorems for the two-stage estimation. We first bound the estimation error in Stage I, then characterize how large attacks must be to be detected by the filtering procedure, and finally derive the estimation error bound of the least-squares estimator in Stage II.

\medskip
\subsubsection{Analysis of Stage I} \label{sec:stage1}
 We now present the main theorem that establishes an estimation error bound using the row-wise $\ell_1$-norm estimators given in \eqref{stage1}, where the outputs are $\mathring a_1, \dots, \mathring a_n$. The proofs in this section are detailed in Appendix \ref{proofvb1}. Before presenting the theorem, we define the universal noise level $\sigma_e$ as a uniform $\psi_2$-norm bound on the noise components $e_t^{(i)}$. Since the composite disturbance $w_t^{(i)} = e_t^{(i)} + k_t^{(i)}$ is sub-Gaussian with a norm of $\sigma$ under Assumption \ref{as:subg}, there exists a constant $\sigma_e \le \sigma$ such that $\|e_t^{(i)}\|_{\psi_2} \le \sigma_e$ for all $t\ge 0$ and $i\in\{1,\dots,n\}$.

\begin{theorem}\label{thm:stage1est}
    Consider Scenario \ref{as:composition}, $ x_{t+1} = \bar A \phi(x_t) + e_t + k_t$. Suppose that Assumptions \ref{as:stability}, \ref{as:subg}, and \ref{as:excitation} hold. Define $\sigma_e$ as the universal noise level. 
     Given $\delta\in (0,1]$, 
     when 
       \begin{align}\label{timebound}           &  T=\Omega\biggr(\frac{(\sqrt{n}L\sigma)^4}{\lambda^4(1-2p) } \\\nonumber&\hspace{5mm} \times \max\biggr\{   \frac{p(\sqrt{n}L\sigma)^{6}\log^2(\frac{n}{\delta})}{(1-2p)\lambda^{6}(1-\rho L)^2 } , m\log\Bigr( \frac{nL\sigma }{\lambda(1-\rho L)}\Bigr) \biggr\}  \biggr),
       \end{align}
it holds that
    \begin{align}\label{upperbound:stage1}
        \|\mathring a_i - \bar a_i \| \le \kappa \frac{ n^2 L^4 \sigma^4\cdot \sigma_e}{(1-2p) \lambda^5}, \quad \forall i=1,\dots,n 
    \end{align}
    with probability at least $1-\delta$, where $\kappa>0$ is a constant.
\end{theorem}

 To derive this result, it is useful to first consider a noise-aware system in which the full noise sequence $\{e_t\}_{t\ge 0}$ is available to the estimator. In such a system, for every $i$, the following lemma obtains a positive  gap between the objective value at the true system $\bar a_i$ and that at any other $a_i \in \mathbb{R}^{n}$.

    \begin{lemma}\label{thm:noiseaware}
        Consider Scenario \ref{as:composition}, $ x_{t+1} = \bar A \phi(x_t) + e_t + k_t$. Let $\tilde x_{t+1} =  x_{t+1}-e_t$. Suppose that Assumptions \ref{as:stability}, \ref{as:subg}, and \ref{as:excitation} hold.  Let
          $\tilde f^{(i)} (a_i) := \sum_{t=0}^{T-1} |\tilde x_{t+1}^{(i)} - a_i^T \phi(x_t) |$
        for every $i\in\{1,\dots,n\}$.  
       Given $\delta\in (0,1],$ 
         suppose that $T$ satisfies \eqref{timebound}. 
    Then, there exists a constant $c>0$ such that
\begin{equation}\label{positivegap}
\begin{split}
    \tilde f^{(i)} (a_i)- \tilde f^{(i)} (\bar a_i) \ge c T\cdot &\frac{(1-2p)\lambda^5 }{n^2 L^4\sigma^4 } \|a_i - \bar a_i\|_2, \\&\forall a_i \in \mathbb{R}^m, \quad \forall i\in\{1,\dots, n\}
    \end{split}
\end{equation}
with probability at least $1-\delta$.
    \end{lemma}

The proof closely follows Lemma~\ref{lem:excitation}, with the difference that the presence of noise allows us to replace Assumption~\ref{as:excitation_attack} with Assumption~\ref{as:excitation}.

\medskip

\subsubsection{Analysis of Filtering Procedure}\label{sec:filtering}

For convenience, define $\tau:= \frac{\kappa n^2 L^4\sigma^4}{(1-2p)\lambda^5}$. 
From Theorem \ref{thm:stage1est}, we obtain a universal upper bound on $\|\mathring a_i-\bar a_i\|_2\le \tau \sigma_e$ with high probability. 
Given the information we have, a single trajectory $(x_0, \dots, x_{T})$ and $\{\mathring a_i^T\}_{i=1}^n$, we can bound $|x_{t+1}^{(i)} - \mathring a_i^T\phi(x_t)|=  |(\bar a_i- \mathring a_i)^T \phi(x_t) + e_t^{(i)} + k_t^{(i)}|$ for all $i$ with high probability.
The next lemma analyzes the filtering procedure in Algorithm~\ref{algo1}. The proof is given in Appendix \ref{proofvb1}.

\begin{lemma}\label{thm:detect}
Consider Scenario \ref{as:composition}. Suppose that Assumptions \ref{as:stability}, \ref{as:subg}, and \ref{as:excitation} hold. Let $\tau:= \frac{\kappa n^2 L^4\sigma^4}{(1-2p)\lambda^5}$. In Algorithm~\ref{algo1}, consider $\beta_1 = \alpha_1\tau\sigma_e$ and $\beta_2 = \alpha_2 \sigma_e$, where $\alpha_1\ge 1$ and $\alpha_2>1$. Given $\delta\in (0,1]$, suppose $T$ satisfies \eqref{timebound}. Then, we obtain 
\begin{align*}
|e_t^{(i)} + k_t^{(i)}|\le [(1+\alpha_1)\tau\|\phi(x_t)\|_2& + \alpha_2]\cdot \sigma_e, \\&\forall i\in\{1,\dots,n\},~~\forall t\in\mathcal{T}_i
\end{align*}
with probability at least $1-\delta$.
\end{lemma}

\begin{remark}
The conservative choice of $\beta_1$ and $\beta_2$ in Algorithm \ref{algo1} reduces false positives, ensuring that clean data are largely preserved. If so, a non-negligible number of attacked data points may remain in the filtered set; however, we accept the resulting bias which is bounded with high probability by Lemma \ref{thm:detect}. This limitation can only be fully overcome when attacked and clean data are perfectly separable, allowing all attacked data to be  detected and filtered out. 
\end{remark}

\medskip
\subsubsection{Analysis of Stage II}\label{sec:stage2}

In this section, we derive an estimation error bound for the least-squares estimator applied to the filtered dataset $\mathcal{T}_i$ for each $i$. Under perfect separability (\textit{e.g.}, attacks $\gg$ noise), the filtered dataset contains only clean data with zero-mean noise and avoids  discarding any tail-noise samples,  allowing the least-squares estimator to achieve the optimal estimation error that converges to zero as $T\to\infty$. However, each misclassification, either a false positive or a false negative, introduces additional bias by discarding clean data or including attacked data. 
We now formally present the main theorem, where $\tilde a_1, \dots, \tilde a_n$ denote the estimates from Algorithm~\ref{algo1}.

\begin{theorem}\label{thm:twostagefinal}
Consider Scenario \ref{as:composition}. Suppose that Assumptions \ref{as:stability}, \ref{as:subg}, and \ref{as:excitation} hold. 
Let $\Gamma_i$ be the time index set corresponding to misclassified false negatives (i.e., attacked data but included in $\mathcal{T}_i$). 
Given $\delta\in(0,1]$, suppose that 
$T$ satisfies \eqref{timebound}.
For Algorithm \ref{algo1} under the same parameter choices as in Lemma~\ref{thm:detect}, the relation
\begin{align*}
    &\|\tilde a _i - \bar a_i\|_2 \\&
=\mathcal{O}\Biggr( \sqrt{ \frac{m}{T}\log\left( \frac{n L\sigma}{\lambda (1-\rho L)\delta} \right)   }\cdot \frac{\sigma}{\lambda } \\&\hspace{5.5mm}+ \sigma_e\biggr(\frac{(1+\alpha_1)\tau \sum_{t\in \Gamma_i} \|\phi(x_t)\|_2^2  + \alpha_2 \sum_{t\in \Gamma_i} \|\phi(x_t)\|_2  }{ T\lambda^2} \\&\hspace{26mm}+   \frac{\sqrt{n}L\sigma}{(\alpha_2-1)(1-\rho L)\lambda^2}\log\left(\frac{n}{\delta}\right)\biggr) \Biggr)\numberthis\label{twoesterr}
\end{align*}
holds for all $i\in\{1,\dots,n\}$ with probability at least $1-\delta$.
\end{theorem}
\textbf{Proof Sketch. } Let $\mathbf{w}_i\in\mathbb{R}^{ |\mathcal{T}_i|}$ denote the vector whose entries are $\{e_t^{(i)}+k_t^{(i)}\}_{t\in \mathcal{T}_i}$, and $\mathbf{X}_i\in\mathbb{R}^{m\times |\mathcal{T}_i|}$ denote the matrix whose rows are $\{\phi(x_t)^T\}_{t\in \mathcal{T}_i}$. Let $\mathbf{\bar w}_i$ denote the conditional-mean vector whose entries are $\{\mathbb{E}[e_t^{(i)}+k_t^{(i)}\;|\; \mathcal{F}_t]\}_{t\in \mathcal{T}_i}$. Then, the least-squares method yields 
\begin{align*}
 &  \|\tilde a _i - \bar a_i\|_2 = \|(\mathbf{w}_i^T\mathbf{X}_i)(\mathbf{X}_i^T\mathbf{X}_i)^{-1}\|_2\\&\hspace{3mm}\le \underbrace{\|((\mathbf{w}_i-\mathbf{\bar w}_i)^T\mathbf{X}_i)(\mathbf{X}_i^T\mathbf{X}_i)^{-1}\|_2}_{(a)}+\underbrace{\|\mathbf{\bar w}_i^T\mathbf{X}_i(\mathbf{X}_i^T\mathbf{X}_i)^{-1}\|_2}_{(b)},
\end{align*}
where $\mathbf{X}_i^T\mathbf{X}_i$ is positive definite with high probability for sufficient large $|\mathcal{T}_i|$ by Lemma \ref{xxlowerbound}. For the term (a), the entries of $\mathbf{w}_i-\mathbf{\bar w}_i$ form a Martingale difference sequence, which allows us to apply the idea in Section D.2 of \cite{simchowitz2018learning} to establish an estimation bound $\mathcal{O}(1/\sqrt{|\mathcal{T}_i|})$. 

For term (b), we consider two cases to bound each entry of $\mathbf{\bar w}_i$. First, for attacked data included in $\mathcal{T}_i$, the quantity $e_t^{(i)} + k_t^{(i)}$ is deterministically bounded by Lemma~\ref{thm:detect}. 
Second, for clean data in $\mathcal{T}_i$, the threshold rule removes the tail of noise and generates a truncation bias, which provides a bound on $\|\mathbf{\bar w}_i^T\mathbf{X}_i\|$. Furthermore, Lemma \ref{xxlowerbound} establishes a bound for $\Vert{}(\mathbf{X}_i^T\mathbf{X}_i)^{-1}\Vert{}_2$. The proof is then completed by leveraging the fact that $\vert{}\mathcal{T}_i \vert{}= \Omega(T)$, with all aforementioned events holding jointly with high probability. \hfill $\blacksquare$

We use the above notation to delineate the lemmas required for the proof.
Detailed proofs for these lemmas and the theorem are deferred to Appendix \ref{proofvb3}. In sequence, the lemmas provide probabilistic upper bounds on
$\|\mathbf{X}_i^T \mathbf{X}_i\|_2$,  $\|(\mathbf{w}_i-\mathbf{\bar w}_i)^T \mathbf{X}_i(\mathbf{X}_i^T \mathbf{X}_i)^{-1}\|_2$, and $\|\mathbf{\bar w}_i^T\mathbf{X}_i\|_2$.

\begin{lemma}\label{xxupperbound}
     Suppose that Assumptions \ref{as:stability} and \ref{as:subg} hold.  Given $\delta\in(0,1]$, we have $\|\mathbf{X}_i^T\mathbf{X}_i\|_2 =\mathcal{O}\left(|\mathcal{T}_i|\bigr(\frac{\sqrt{n} L\sigma}{1-\rho L}\bigr)^2 \log\bigr(\frac{1}{\delta}\bigr) \right) $ with probability at least $1-\delta$.
\end{lemma}

\begin{lemma}\label{xylowerbound}
    Suppose that Assumptions \ref{as:stability}, \ref{as:subg}, and \ref{as:excitation} hold, Given $\delta\in(0,1]$, when $|\mathcal{T}_i|$ satisfies \eqref{finaltimebound}, we have 
    \begin{align*}
      &  \|(\mathbf{w}_i-\mathbf{\bar w}_i)^T \mathbf{X}_i(\mathbf{X}_i^T \mathbf{X}_i)^{-1}\|_2 \\&\hspace{20mm}= \mathcal{O}\left(\frac{\sigma}{\lambda \sqrt{|\mathcal{T}_i|}}\cdot \sqrt{ m \log\left( \frac{n L \sigma}{\lambda (1-\rho L)\delta} \right) } \right)
    \end{align*}
    with probability at least $1-\delta$.
\end{lemma}

\begin{lemma}\label{asdf}
 Let $\Gamma_i$ be the time index set corresponding to misclassified false negatives (i.e., attacked data but included in $\mathcal{T}_i$).     Given $\delta\in(0,1]$, when $T$ satisfies \eqref{timebound}, then we have     
     \begin{equation*}
     \begin{split}
         &\|\mathbf{\bar w}_i^T \mathbf{X}_i\|_2\\&= \mathcal{O}\Biggr((1+\alpha_1)\tau\sigma_e\sum_{t\in \Gamma_{i}}  \|\phi(x_t)\|_2^2 + \alpha_2 \sigma_e \sum_{t\in \Gamma_{i}} \|\phi(x_t)\|_2 \\&\hspace{20mm}+\sigma_e   \frac{\sqrt{n}L\sigma}{(\alpha_2-1)(1-\rho L)}\log\left(\frac{1}{\delta}\right) |\mathcal{T}_i|\Biggr)
         \end{split}
     \end{equation*}
     with probability at least $1-\delta$. 
 \end{lemma}

\begin{remark}
In Theorem \ref{thm:twostagefinal}, the first error term  in \eqref{twoesterr} diminishes at a rate of $O(1/\sqrt{T})$. In contrast, since $|\Gamma_i|=\mathcal{O}(T)$, the second and the third terms in \eqref{twoesterr} remain constant, proportional to the noise level $\sigma_e$ times the number of false detections. This implies that the estimation error diminishes with $T$ for perfectly separable data, which occurs when all attacks are sufficiently large compared to the noise level.  Since misclassifications are driven by attack magnitudes beyond the estimator's control, the focus must be shifted to reducing $\sigma_e$.
However, a reduction in noise often leads to a decrease in the excitation parameter $\lambda$ (c.f., no noise gives no excitation at all), which complicates identification. A robust remedy is to lower the small noise floor via improvements in modeling accuracy, while simultaneously injecting a known control input to ensure persistent excitation. For instance, consider a system with additive control:
    $x_{t+1} = \bar{A}\phi(x_t) + u_t + e_t + k_t$,
where $u_t$ is an independent zero-mean Gaussian control input. The $u_t$ term provides additional excitation on top of excitation by $e_t$, which enables to identify $\bar{A}$ via two-stage estimation by treating $x_{t+1}-u_t$ as the next state.
This approach yields an estimation error that decreases proportionally as the noise parameter $\sigma_e$ is reduced.
\end{remark}

\section{Numerical Experiments}\label{sec:numexp}
In this section, we provide experimental validations tested on the discrete-time dynamical systems to demonstrate our theoretical findings in the previous sections. 

\vspace{-2mm}
\subsection{Agnostic Robustness}
We consider the state dimension $n=10$, and   
\begin{align*}
\phi(x) = [
 x^{(1)}&, x^{(2)}, x^{(3)}, x^{(4)} \tanh(x^{(5)}), x^{(5)} \tanh(x^{(6)}),\\& x^{(6)} \tanh(x^{(4)}), \
\sin((x^{(7)})^2), \cos((x^{(8)})^2), \\&\sin((x^{(9)})^2), \sin(x^{(1)} x^{(2)}), \cos(x^{(10)})
]^T \in \mathbb{R}^{11}.
\end{align*}
We randomly generate the true system matrix $\bar A \in \mathbb{R}^{10 \times 11}$ and deliberately sparsify it by setting a subset of its entries to zero. This enforces a sparse structure within the expressive nonlinear feature space $\phi(x)$ to incorporate the SINDy framework \cite{brunton2016discovering}. The non-zero entries are then scaled such that the spectral norm of the system matrix satisfies $\|\bar A\|_2 = 0.95$.
The true trajectory of the system is generated from
\begin{align*}
    x_0 = [1,\dots,1]^T, \quad x_{t+1} = \bar A \phi(x_t) + w_t, ~~t=0,\dots, T-1,
\end{align*}
where $T=2000$.  We consider two  disturbance scenarios for $w_t \in \mathbb{R}^{10}$: (a) a symmetric case where each component is independently uniform on $[-2, 2]$, and (b) a sparse case where with probability $0.6$, $w_t$ equals the zero vector, and with probability $0.4$, its components are uniformly distributed on $[2-\min\{\|x_t\|_2, 10\},~ 2+\min\{\|x_t\|_2, 10\}]$, deliberately designed to depend on $x_t$. Under each of these disturbances, we run \eqref{ls}, \eqref{l1}, and \eqref{hubermin} with $\mu=0.75$ to obtain estimates $\hat A$. 

\begin{figure}[t]
     \centering
     \begin{subfigure}[b]{0.24\textwidth}
         \centering
    \includegraphics[width=\linewidth]{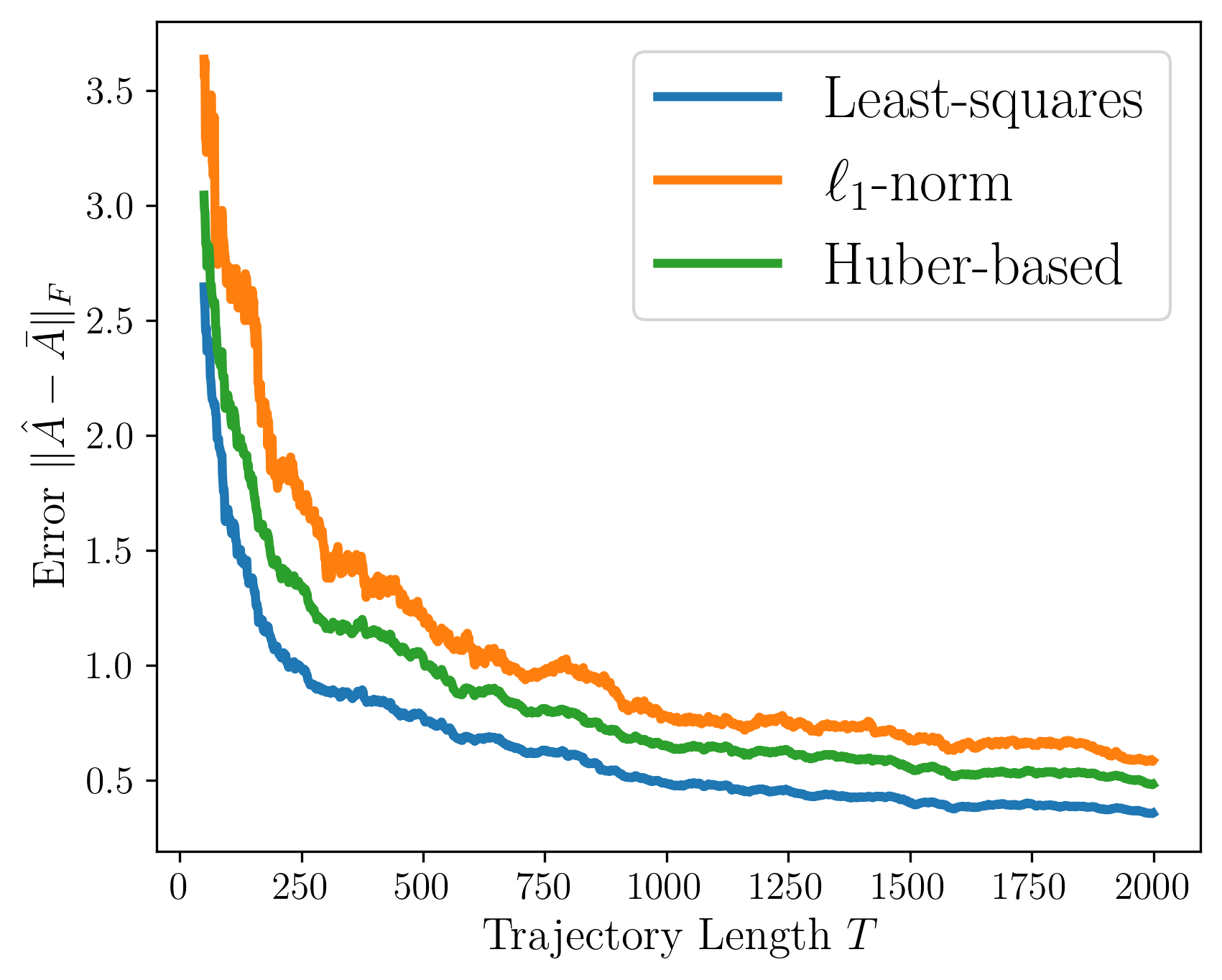}
         \caption{Persistent zero-mean noise}
         \label{fig:zerofrob}
     \end{subfigure}
     \begin{subfigure}[b]{0.24\textwidth}
         \centering
    \includegraphics[width=\linewidth]{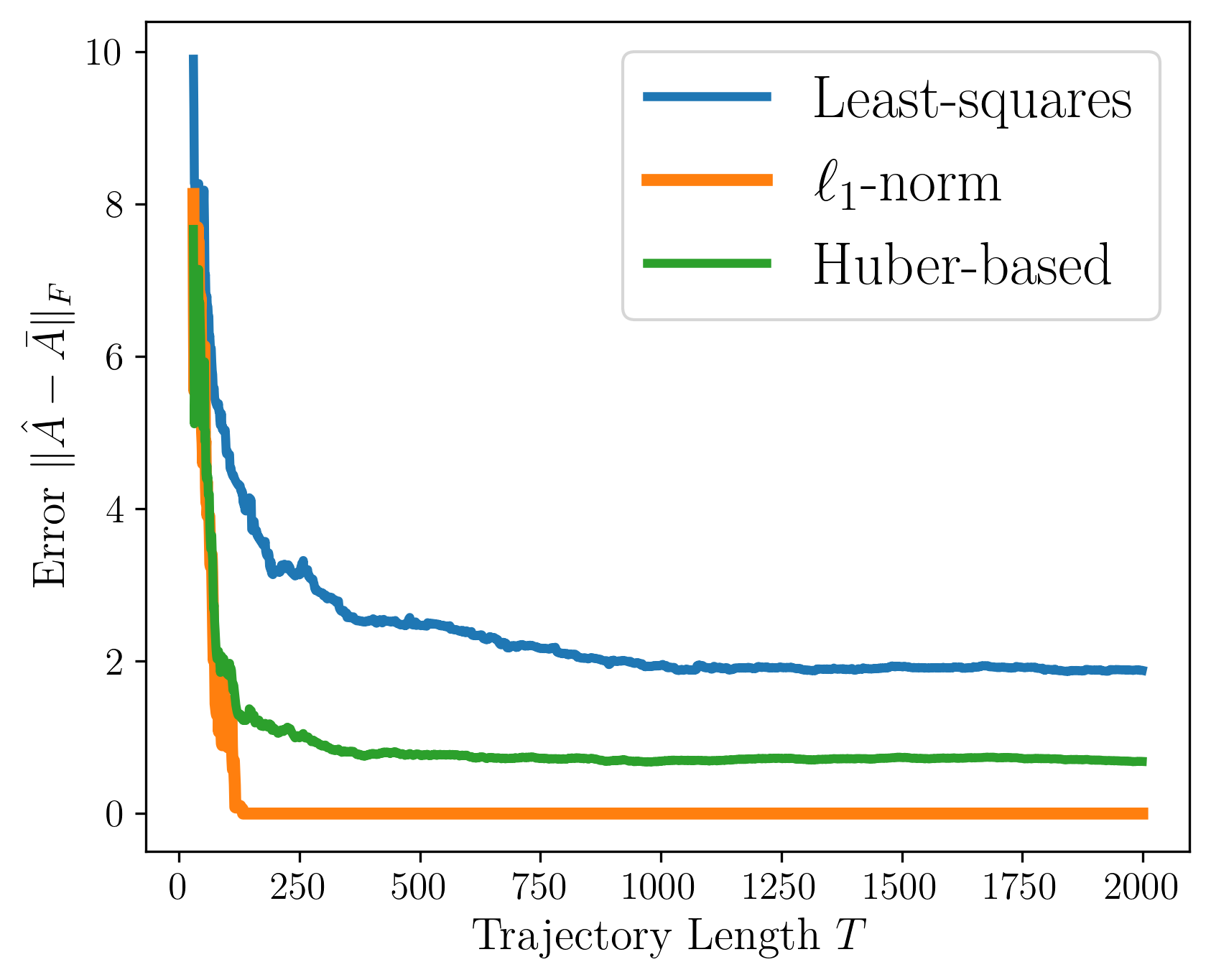}
         \caption{Sparse nonzero-mean attack}
         \label{fig:nonzerofrob}
     \end{subfigure}
        \caption{Agnostic Robustness: Estimation error of the Huber estimator across pure noise and pure attack regimes.}
        \label{ex2}
        \vspace{-2mm}
\end{figure}

Figure \ref{ex2} validates the theoretical error bounds derived in Theorems \ref{noisescenario} and \ref{attackscenario} by plotting the Frobenius error norm $\|\hat A - \bar A\|_F$, against the trajectory length $T \in [50, 2500]$. Under persistent zero-mean noise (Figure \ref{fig:zerofrob}), the least-squares estimator converges at a rate of $\mathcal{O}(1/\sqrt{T})$. The Huber estimator matches this rate up to a constant factor, consistently yielding tighter estimation error than the $\ell_1$-norm estimator. Under sparse nonzero-mean attack (Figure \ref{fig:nonzerofrob}), the $\ell_1$-norm estimator perfectly recovers the system with zero error for $T \ge 130$. The Huber estimator yields a bounded constant error of $\mathcal{O}(\mu)$, significantly outperforming the least-squares approach. These results confirm that the Huber estimator, with an appropriate value of $\mu$, serves as a robust bridge between standard mean- and median-based estimators across pure noise and  pure attack regimes.

\subsection{Composite Robustness}

\begin{figure}[t]
     \centering
     \begin{subfigure}[b]{0.24\textwidth}
         \centering
    \includegraphics[width=\linewidth]{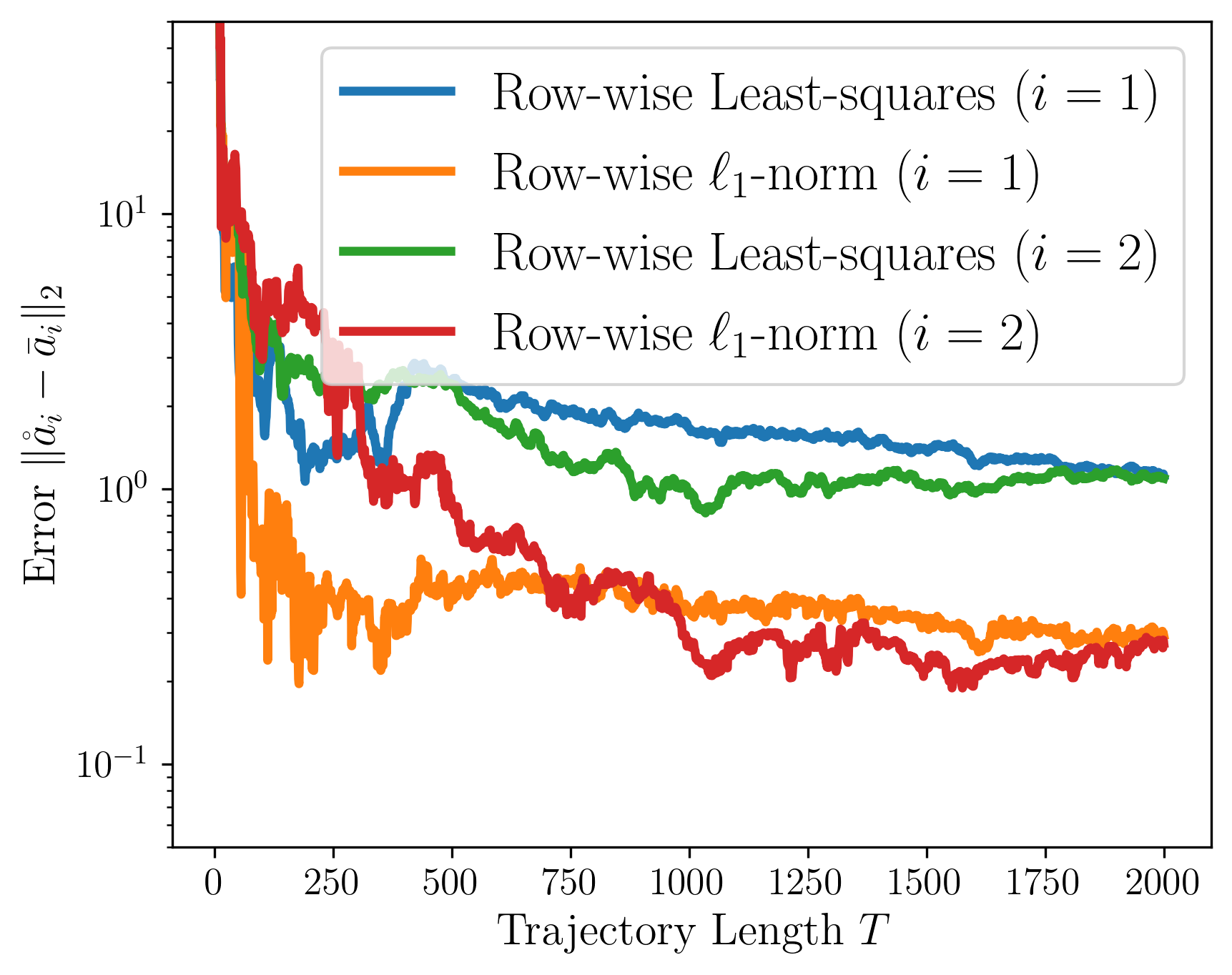}
         \caption{Comparison between one-stage estimators}
         \label{fig:l1}
     \end{subfigure}
     \begin{subfigure}[b]{0.24\textwidth}
         \centering
    \includegraphics[width=\linewidth]{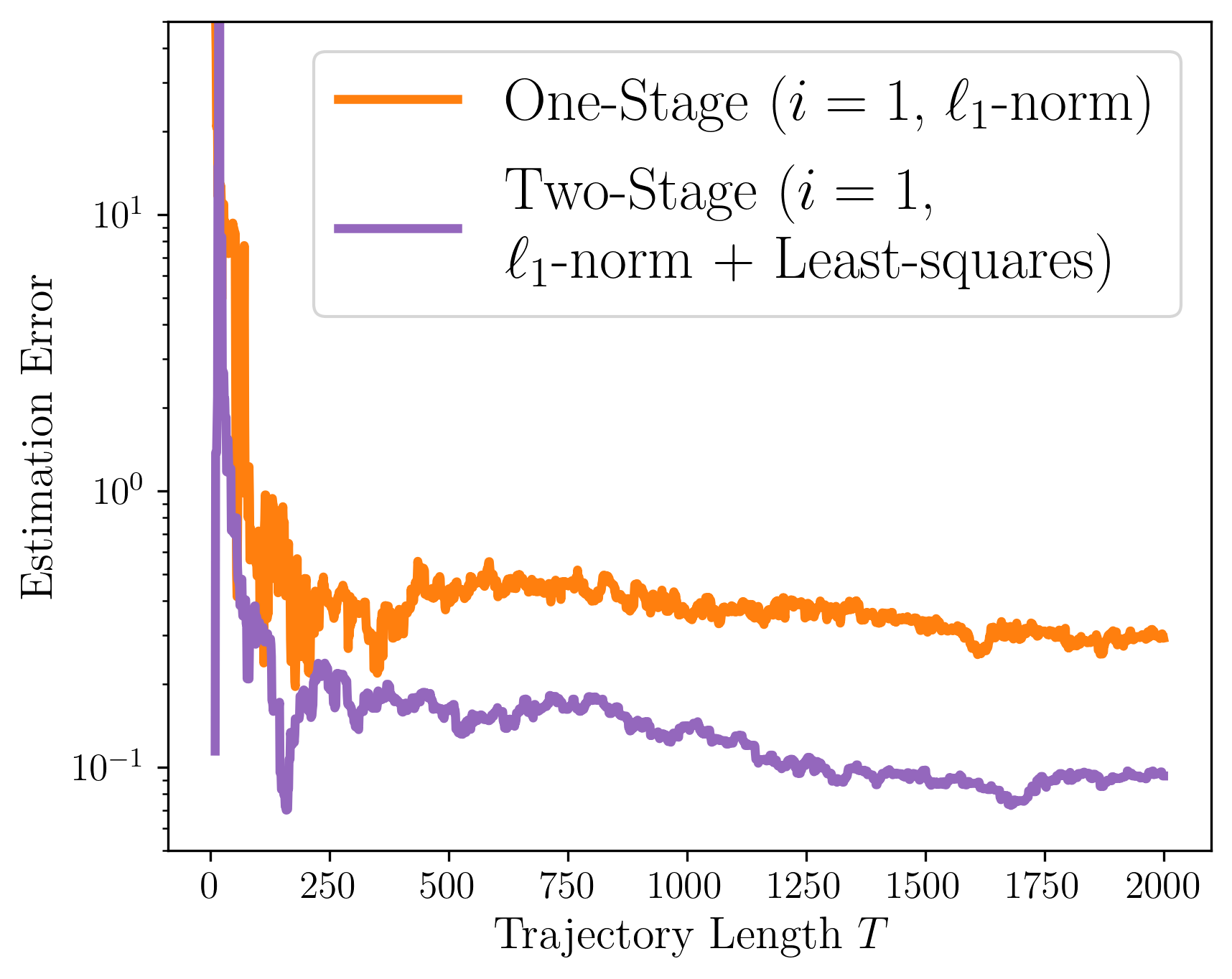}
         \caption{One-stage estimator vs. Two-stage Estimation}
         \label{fig:twostage}
     \end{subfigure}
        \caption{Composite Robustness: Estimation error of the two-stage method under concurrent noise and attacks. (a) For both $i=1,2$, the row-wise $\ell_1$-norm estimator achieves a low error for Stage I. (b) Post-filtering Stage II estimation outperforms one-stage approaches.}
        \label{fig:experiments}
\end{figure}

The system dynamics are defined the same as in the agnostic robustness experiments.  We consider $w_t = e_t + k_t$, where  $e_t$ is the persistent noise drawn from a Gaussian distribution with zero mean and a standard deviation of $1$,  and $k_t$ is the
coordinate-wise adversarial attack (see Remark \ref{coordinate}) with attack probability $p=0.4$. To amplify the misleading effect, the adversary sets the attack signal $k_t^{(i)}$ to $20 \cdot \text{sgn}(x_t^{(i)})$ whenever the $i$-th coordinate is being attacked. 

In the first experiment, we compare two one-stage estimators: the row-wise least-squares and the row-wise $\ell_1$-norm estimator. We report the estimation error for the first two rows of $\bar A$ ($\bar a_1$  and $\bar a_2$) against the trajectory length $T$ on a logarithmic scale (see Figure~\ref{fig:l1}). The results support the absolute superiority of the $\ell_1$-norm estimator; for $T\ge 500$, the  $\ell_1$-norm estimators consistently yield lower errors than the least-squares methods across both $i=1,2$.
 This experiment justifies the selection of the $\ell_1$-norm estimator for Stage I of Algorithm \ref{algo1}, since it yields a low baseline error, a constant error floor which is unavoidable for one-stage estimators as proven in Section \ref{sec:onestage}.

For the second experiment, we take the initial estimate $\mathring a_i$ for each coordinate $i$ from 
the row-wise $\ell_1$-norm estimators evaluated at $T=2000$ in the first experiment.
 Since the attack magnitude is significantly larger than the Gaussian noise, the dataset shows clear separability. 
 We construct a filtered dataset comprising only the points that satisfy the threshold rule \eqref{filter} with threshold parameters $\beta_1= 1$ and $\beta_2=3$. Subsequently, we apply the least-squares algorithm to the filtered data and report the estimation error for $\bar a_1$ versus trajectory length on a logarithmic scale (see Figure \ref{fig:twostage}). We use the row-wise $\ell_1$-norm estimator as the baseline, since it is the top-performing one-stage estimator in Figure \ref{fig:l1}. Note that the error of the two-stage estimator at time $T$ is computed using the filtered data up to time $T$. 
 The result shows that the two-stage estimator outperforms the one-stage baseline, confirming that our procedure effectively filters out attacks while preserving clean data.

\section{Conclusion}\label{sec:conclusion}

In this paper, we introduce a robust system identification framework for a single trajectory of length $T$, based on 
two different robustness notions: agnostic and composite robustness. The former identifies a system without knowing whether a pure noise or pure attack regime is active; the  latter recovers a system subject to concurrent noise and attacks.
For agnostic robustness,
we prove that the Huber estimator achieves the optimal $\mathcal{O}(1/\sqrt{T})$ error rate under persistent zero-mean noise when the noise is symmetric or the basis functions are linear, while ensuring a bounded constant estimation error against sparse adversarial attacks.
Then, under the composite regime, we show that any convex one-stage estimator fails to achieve an estimation error that converges to zero as $T \to \infty$. To overcome this fundamental limitation and achieve composite robustness, we propose a two-stage estimation method that leverages row-wise $\ell_1$-norm estimators in the first stage to filter out time instances dominated by large attacks, followed by the least-squares method on the retained data with manageable noise levels. The resulting estimation error bound is the sum of $\mathcal{O}(1/\sqrt{T})$ rate and a constant term proportional to the product of the noise level and the number of false detections. 
This work provides the first unified theoretical guarantees for modern real-world systems subject to diverse disturbance scenarios, revealing a fundamental duality:  \textit{separate} occurrences of noise or attacks can be overcome by \textit{aggregating} median- and mean-based estimators into a single objective, whereas \textit{aggregated} noise and attacks must be overcome through the \textit{separate}, sequential application of these methods.

\appendix

\subsection{Proof of Theorem \ref{noisescenario}}\label{hubernoise}

\begin{proof}
 Since $H_\mu$ is convex, the first-order conditions provide necessary and sufficient conditions for optimality of \eqref{hubermin}.

 Then, considering that $x_{t+1}^{(i)} = \bar a_i^T \phi(x_t) +w_t^{(i)}$,  
and denoting $\bar \epsilon_i = \bar a_i - \hat a_i$ for each $i$, we have 
 \begin{align}\label{firstorder}
     \sum_{t=0}^{T-1}  H_{\mu}' ( \bar \epsilon_i^T \phi(x_t) + w_t^{(i)} )\cdot  \phi(x_t) = 0 , ~~ \forall i\in\{1,\dots,n\},
 \end{align}
 where $H_\mu'(z)$ is defined in \eqref{huberderivative}. 
 Let 
 \begin{align*}
     F^{(i)} (\epsilon) :=  \sum_{t=0}^{T-1}  H_{\mu}' (  \epsilon^T \phi(x_t) + w_t^{(i)} )\cdot  \phi(x_t).
 \end{align*}
 We define the following time index set  
\begin{align*}
   & S_i(\epsilon) = \{t\in\{0,\dots, T-1\}: |\epsilon^T \phi(x_t) + w_t^{(i)}|\le \mu\}. 
\end{align*}
Then, the Jacobian of $F^{(i)} (\epsilon)$ is 
defined as 
    $J^{(i)}(\epsilon) = \sum_{t\in S_i(\epsilon) }\phi(x_t)\phi(x_t)^T,$
since the second derivative of $H_\mu(z)$ is $1$ if $|z|\le \mu$ and $0$ otherwise. By the fundamental theorem of calculus, we have
\begin{align}\label{fund}
    F^{(i)} (\epsilon)-F^{(i)} (0) = \int_{0}^1   J^{(i)}(s \epsilon) \epsilon ~ds.
\end{align}

  Let $\Gamma_i = \{t\in\{0,\dots, T-1\}: |w_t^{(i)}| \le \frac{\mu}{2}\}$. 
   We now consider the events
   \begin{align*}
     &  \mathcal{E}_1 = \Bigr\{|\Gamma_i|\ge \frac{qT}{2}\Bigr\}, \\& \mathcal{E}_2 = \bigr\{  |\epsilon^T \phi(x_t)| \le \frac{\mu}{2}, ~~\forall t=0,\dots, T-1\bigr\}
   \end{align*}
Under these two events, we know that there exists a time index set $\tilde{\Gamma}_i(\epsilon)$ with a cardinality of at least $\frac{qT}{2}$, such that for all $t \in \tilde{\Gamma}_i(\epsilon)$, $|w_t^{(i)}| \le \frac{\mu}{2}$ and $|\epsilon^T \phi(x_t)| \le \frac{\mu}{2}$. More importantly, $\tilde{\Gamma}_i(\epsilon)$ is a subset of $S_i(s\epsilon)$ for any $0 \le s \le 1$.
 Then, we consider the third event
\begin{align*}
    \mathcal{E}_3 = \Bigr\{  \sum_{t\in \tilde \Gamma_i(\epsilon)} \phi(x_t) \phi(x_t)^T \succeq \frac{\lambda^2 I}{2} \cdot \frac{qT}{2}  \Bigr\}.
\end{align*}
Under this event, we know that $J^{(i)}(s\epsilon)\succeq \frac{\lambda^2 qT }{4}I$ for any $0\le s\le 1$. 
By multiplying $\epsilon^T$ to both sides of \eqref{fund}, we obtain 
\begin{align*}
    \epsilon^T F^{(i)}(\epsilon) -\epsilon^T& F^{(i)}(0) = \int_0^1 \epsilon^T J^{(i)}(s\epsilon) \epsilon~ ds\\&\ge \int_0^1 \frac{\lambda^2 qT }{4}\|\epsilon\|_2^2 ~ds =\frac{\lambda^2 qT }{4}\|\epsilon\|_2^2.\numberthis\label{epseps}
\end{align*}
Now, we will measure the probability that the events $\mathcal{E}_1, \mathcal{E}_2, \mathcal{E}_3$ hold simultaneously. 
For $\mathcal{E}_1$, one can apply the Chernoff bound under Assumption \ref{as:mildinnoise} to obtain $\mathbb{P}(|\Gamma_i| \ge \frac{qT}{2})\ge 1-\exp(-\frac{qT}{8})$, which implies that 
 $T=\Omega\bigr( \frac{1}{q} \log(\frac{1}{\delta})\bigr)$ ensures that $\mathbb{P}(\mathcal{E}_1)\ge1-\frac{\delta}{6}$.

For $\mathcal{E}_2$, we construct a union bound over all $t$ to obtain $\max_{t\in\{0,\dots,T-1\}} \|\phi(x_t)\|_2  = \mathcal{O} \bigr(\frac{\sqrt{n} L\sigma}{1-\rho L}\sqrt{\log(\frac{T}{\delta})}\bigr)$ as
\begin{align*}
       & \mathbb{P}\bigr(\max_{t\in\{0,\dots, T-1\}}\|\phi(x_t)\|_2\ge s\bigr) \\&= \mathbb{P}\bigr(\{\|\phi(x_0)\|_2\ge s\}\cup \dots \cup \{\|\phi(x_{T-1})\|_2\ge s\}\bigr)\\&\le T \cdot 2\exp\Bigr(-\Omega\Bigr(\frac{s^2}{(\sqrt{n}L\sigma  / (1-\rho L))^2}\Bigr)\Bigr),
    \end{align*}
    for all $s\ge 0$. Considering that $|\epsilon^T \phi(x_t)|\le \|\epsilon\|_2 \|\phi(x_t)\|_2$, 
there exists a constant $c_7>0$ such that 
\begin{align*}\label{epsR}
   & \|\epsilon\|_2 \le \frac{c_7\mu (1-\rho L)}{\sqrt{n}L \sigma\sqrt{\log(T/\delta)} }:=R ~~\Longrightarrow  ~~\mathbb{P}\left(\mathcal{E}_2 \right)\ge 1-\frac{\delta}{6}.
\end{align*}

For $\mathcal{E}_3$, since $|\tilde \Gamma_i(\epsilon)|\ge\frac{qT}{2}$ under $\mathcal{E}_1\cap \mathcal{E}_2$, 
when $\frac{qT}{2}$ satisfies the time complexity \eqref{finaltimebound}, $\mathcal{E}_3$ holds with probability at least $1-\frac{\delta}{6}$. Taking the union bound, 
when $\|\epsilon\|_2\le R$ and $\frac{qT}{2}$ satisfies the time complexity \eqref{finaltimebound} (which already subsumes $T=\Omega(\frac{1}{q}\log(\frac{1}{\delta}))$),
we have $\mathbb{P}(\mathcal{E}_1\cap \mathcal{E}_2 \cap \mathcal{E}_3)\ge 1-\frac{\delta}{2}$; as a result, \eqref{epseps} holds with probability at least $1-\frac{\delta}{2}$. 

Meanwhile, $\|F^{(i)}(0)\|_2$ is bounded by $\mathcal{O}(\sqrt{T})$ with probability at least $1-\frac{\delta}{2}$, which follows from Lemma 
  \ref{secondterm}. In this case, large enough $T$ ensures that $\frac{\lambda^2 qT}{4}R -\|F^{(i)}(0)\|_2 > 0$. In particular, using the fact that $T \ge 2A \log(A/\delta)$ implies  $\frac{T}{\log(T/\delta)} > A$, it suffices to have
 \begin{align*}
    & T = {\Omega} \biggr( \frac{(\sqrt{n} L\sigma)^4 }{ q^2\lambda^4 (1-\rho L)^4} \log^2\left(\frac{m}{\delta}\right)\log\left( \frac{Ln\sigma }{q\lambda (1-\rho L) \delta} \right) \biggr)\numberthis\label{timebound4}
 \end{align*}
 for $\frac{\lambda^2 qT}{4}R -\|F^{(i)}(0)\|_2 > 0$ to hold  with probability at least $1-\frac{\delta}{2}$. In such a case, as a by-product, we have $R=\mathcal{O}\bigr(\frac{ \mu  \sqrt{n}L\sigma}{\sqrt{T}q\lambda^2 (1-\rho L)}\log\bigr(\frac{1}{\delta}\bigr)\bigr).$
Note that \eqref{timebound4} implies that 
$\frac{qT}{2}$ satisfies the time complexity \eqref{finaltimebound}. Thus, we can use the union bound to establish that
 $\|\epsilon\|_2 =R$ in \eqref{epseps} implies that 
\begin{equation}\label{eps0}
\begin{split}
\epsilon^T F^{(i)}(\epsilon)&\ge  \frac{\lambda^2 qT}{4}\|\epsilon\|_2^2 - \|\epsilon\|_2 \|F^{(i)}(0)\|_2 \\&= R\Bigr(\frac{\lambda^2 qT}{4}R -\|F^{(i)}(0)\|_2 \Bigr)>0.
\end{split}
\end{equation}
  with probability at least $1-\delta$  under \eqref{timebound4}.

Note that $\bar \epsilon_i$ defined in \eqref{firstorder} cannot satisfy \eqref{eps0} since $F^{(i)}(\bar \epsilon_i)=0$. Meanwhile, by the continuity of $F^{(i)}(\epsilon)$,   $\|\epsilon\|_2 = R \Rightarrow  \epsilon^T F^{(i)}( \epsilon)>0$ implies that there exists $\|\epsilon\|_2<R$ such that $F^{(i)}( \epsilon)=0$
(see \cite[Theorem 6.3.4]{Ortega2000}). Noting that a set of optimal points to convex optimization problem \eqref{hubermin} is indeed convex, every solution to \eqref{hubermin} should satisfy that $\|\epsilon\|_2 <R$. Under \eqref{timebound4}, we have 
\begin{align}\label{epsbound1}
    \|\epsilon\|_2<R= \mathcal{O}\biggr(\frac{ \mu  \sqrt{n}L\sigma}{\sqrt{T}q\lambda^2 (1-\rho L)}\log\Bigr(\frac{1}{\delta}\Bigr)\biggr)
\end{align}
with probability at least $1-\delta$. 

To ensure that the aforementioned argument holds for all $i\in\{1,\dots,n\}$, we substitute $\frac{\delta}{n}$ for $\delta$ in \eqref{timebound4} and \eqref{epsbound1} to  complete the proof.
\end{proof}

\subsection{Proof of Theorem \ref{attackscenario}}\label{huberattack}
\begin{proof}
    Let
$ f^{(i)}(a_i) := \sum_{t=0}^{T-1} |x_{t+1}^{(i)}-a_i^T \phi(x_t)|$ and   $h^{(i)}(a_i):= \sum_{t=0}^{T-1} H_\mu (x_{t+1}^{(i)}-a_i^T \phi(x_t) )$.

 Note that we have
    \begin{align*}
     \mu |z|-  H_\mu(z)  = \begin{dcases*}
         \mu|z| - \frac{1}{2}z^2 & if $|z|\le \mu$,\\ \frac{1}{2}\mu^2 & if $|z| > \mu$,
     \end{dcases*} 
    \end{align*}
    where $0\le \mu|z| - \frac{1}{2}z^2 \le \frac{1}{2}\mu^2$ if $|z| \le \mu$. This implies that 
    \begin{align*}
     0\le    \mu f^{(i)}(a_i) - h^{(i)}(a_i) \le \frac{\mu^2 T}{2}, \quad \forall a_i\in \mathbb{R}^m.
    \end{align*}
    Then, for every $i\in\{1,\dots,n\}$,  we obtain the relationship
    \begin{align*}
         \mu f^{(i)}(\hat a_i) \le h^{(i)}(\hat a_i) +  \frac{\mu^2 T}{2} &\le h^{(i)}(\bar a_i) +  \frac{\mu^2 T}{2}\\&\le \mu f^{(i)}(\bar a_i) +  \frac{\mu^2 T}{2},
    \end{align*}
where the second inequality follows from the optimality of $\hat a_i$ to \eqref{hubermin} for every $i$. This quantifies an upper bound on  $\mu f^{(i)}(\hat a_i)-\mu f^{(i)}(\bar a_i)$.  Lemma \ref{lem:excitation} implies that under \eqref{timeboundattack}, there exists a constant $c>0$ such that
\begin{align*}
  c\mu T &\frac{p(1-2p)\lambda^5 }{n^2 L^4\sigma^4 } \|\hat a_i - \bar a_i\|_2  \le \mu f^{(i)}(\hat a_i)-\mu f^{(i)}(\bar a_i) \le  \frac{\mu^2 T}{2}
\end{align*}
  for all $i$, with probability at least $1-\delta$. Rearranging the left-hand  and right-hand sides completes the proof.
\end{proof}

\subsection{Proofs  in Section \ref{sec:onestage}}\label{lemmasproof4}
\subsubsection{Proof of Lemma \ref{borelcantelli}}
    Since the distribution of $e_t$ is symmetric and $h_t(\cdot)$ is an odd function, $h_t(e_t)$ is symmetric in $\mathbb{R}$. Let $\{\varepsilon_t\}_{t \ge 0}$ be a sequence of independent Rademacher random variables ($\mathbb{P}(\varepsilon_t = 1) = \mathbb{P}(\varepsilon_t = -1) = 1/2$). Then, the sequence $\{h_t(e_t)\}_{t \ge 0}$ is identically distributed to $\{\varepsilon_t h_t(e_t)\}_{t \ge 0}$. Also, since the subsampling variable $\xi_t$ is independent of $e_t$, the subsampled sequence $\{(1-\xi_t)h_t(e_t)\}_{t \ge 0}$ has exactly the same distribution as $\{(1-\xi_t) \varepsilon_t h_t(e_t)\}_{t \ge 0}$. By this equivalence in distribution, we may proceed with the proof by incorporating $\varepsilon_t$ into the definitions of $Z_T$ and $\tilde{Z}_T$ without loss of generality.

    Suppose that $Z_T  \xrightarrow{a.s.} 0$. 
  It is  equivalent to stating that for any $\epsilon>0$, $\lim_{m \to \infty} \mathbb{P}\left( \sup_{T \ge m} |Z_T| > \epsilon \right) = 0$.

Fix $N > m \ge 1$. Let $\mathcal{Y} = \mathbb{R}^{N-m+1}$ be a Banach space equipped with the supremum norm, $\|\mathbf{y}\|_\infty = \max_{m \le T \le N} |y_T|$, and standard basis $\{\mathbf{u}_T\}_{T=m}^N$. The trajectory $\mathbf{Z}_{m,N} = (Z_m, \dots, Z_N)$ can be written in $\mathcal{Y}$ by exchanging the order of summation:$$\mathbf{Z}_{m,N} = \sum_{T=m}^N \biggr( \frac{1}{T} \sum_{t=0}^{T-1} \varepsilon_t  h_t(e_t) \biggr) \mathbf{u}_T = \sum_{t=0}^{N-1} \varepsilon_t h_t(e_t) \mathbf{v}_t$$where $\mathbf{v}_t = \sum_{T = \max(m, t+1)}^N \frac{1}{T} \mathbf{u}_T \in E$. 
The subsampled trajectory is identically $\tilde{\mathbf{Z}}_{m,N} = \sum_{t=0}^{N-1} (1-\xi_t)\varepsilon_t h_t(e_t) \mathbf{v}_t$. Conditioned on every system variable except $\{\varepsilon_t\}_{t \ge 0}$, the scalars $\alpha_t = (1-\xi_t)$ satisfy $0 \le \alpha_t \le 1$. By the contraction principle for symmetric random variables in a Banach space \cite[Theorem 4.4]{ledoux1991probability}, deterministic scaling by $\alpha_t \in [0,1]$ at most doubles the tail probability of the norm. Taking the unconditional expectation yields $\mathbb{P}\left( \max_{m \le T \le N} \vert{}\tilde{Z}_T\vert{} > \epsilon \right) \le 2 \mathbb{P}\left( \max_{m \le T \le N} \vert{}Z_T\vert{} > \epsilon \right)$. 
Taking the limit as $N \to \infty$ gives:
$$\mathbb{P}\left( \sup_{T \ge m} \vert{}\tilde{Z}_T\vert{} > \epsilon \right) \le 2 \mathbb{P}\left( \sup_{T \ge m} \vert{}Z_T\vert{} > \epsilon \right)$$Since $Z_T \xrightarrow{a.s.} 0$, the right-hand side converges to $0$ as $m \to \infty$. Consequently, $\lim_{m \to \infty} \mathbb{P}(\sup_{T \ge m} \vert{}\tilde{Z}_T\vert{} > \epsilon) = 0$, proving $\tilde{Z}_T \xrightarrow{a.s.} 0$. 
\hfill $\blacksquare$

\subsubsection{Proof of Lemma \ref{EphiFbounded}}
   We first consider $\|\phi(x_t)\|_2 \le L\|\bar A\phi(x_{t-1})\|_2  + L\|w_{t-1}+v_{t-1}\|_2 + \|\phi(0)\|_2$ for $t\ge 1$ by the Lipschitz property of $\phi$. Then, for any $\epsilon >0$, we have 
\begin{align*}
   \mathbb{E}[ \|\phi&(x_t)\|_2^3 \;|\;\mathcal{F}_{t-1}] \le   (1+\epsilon)^2 (\rho L\|\phi(x_{t-1})\|_2)^3 \\&+ \underbrace{\Bigr(\frac{1+\epsilon}{\epsilon}\Bigr)^2 \mathbb{E}[ (L\|w_{t-1}+v_{t-1}\|_2 + \|\phi(0)\|_2)^3\;|\;\mathcal{F}_{t-1}]}_{e_{t-1}}.
\end{align*}
Define $\gamma := (1+\epsilon)^2 (\rho L)^3 $.
One can select a sufficiently small $\epsilon$ that satisfies $\gamma < 1$ since $\rho L<1$. Then, one can upper bound the term $e_{t-1}$ 
 by a  constant $\tilde C>0$ defined by $\sigma, \rho, L$.  
Then, we have
\begin{align*}
   \mathbb{E}[ \|\phi(x_{t})&\|_2^3 ~|~\mathcal{F}_{t-1}] \le \gamma  \mathbb{E}[ \|\phi(x_{t-1})\|_2^3 ~|~\mathcal{F}_{t-2}] \\&+ \gamma (\underbrace{ \|\phi(x_{t-1})\|_2^3 - \mathbb{E}[\|\phi(x_{t-1})\|_2^3 \mid \mathcal{F}_{t-2}]}_{H_{t-1}}) + \tilde C  
\end{align*}
 A telescoping sum over $t=1,\dots, T-1$ yields
    $ (1-\gamma ) \sum_{t=1}^{T-1}\mathbb{E}[ \|\phi(x_{t})\|_2^3 ~|~\mathcal{F}_{t-1}] \le \gamma  \mathbb{E}[ \|\phi(x_{0})\|_2^3 ] + \gamma \sum_{t=1}^{T-1}H_{t-1} + (T-1)\tilde C,$
 which implies
 \begin{align*}
      \frac{1}{T}\sum_{t=1}^{T-1}\mathbb{E}[ \|\phi(x_{t})\|_2^3 \;|\;&\mathcal{F}_{t-1}] \le \frac{\gamma}{(1-\gamma) T}  \mathbb{E}[ \|\phi(x_{0})\|_2^3 ]\\& + \frac{\gamma}{1-\gamma } \frac{1}{T}\sum_{t=1}^{T-1}H_{t-1} + \frac{T-1}{(1-\gamma )T}\tilde C.
 \end{align*}
By SLLN for Martingales \cite{hall1980martingale}, $\lim_{T\to \infty}\frac{1}{T}\sum_{t=1}^{T-1} H_{t-1}$ converges to zero almost surely; this concludes that 
\begin{align*}
  \limsup_{T\to \infty} \frac{1}{T}\sum_{t=1}^{T-1} \mathbb{E}[\|\phi(x_t)\|_2 ^3\;|\;\mathcal{F}_{t-1} ]\le \frac{\tilde C}{1-\gamma}:= C
\end{align*}
almost surely.
\hfill $\blacksquare$

\subsection{Proofs in Section \ref{sec:stage1}-\ref{sec:filtering}}\label{proofvb1}

\subsubsection{Proof of Lemma \ref{thm:noiseaware}}
Fix $i$ and let $\mathcal{N} = \{t\in \{1,\dots, T-1\}: k_t^{(i)} = 0\}$ and $\mathcal{T}_{non} = \{t\in \{1,\dots, T-1\}: k_t^{(i)} = 0, ~\xi_{t-1} = 1\}$. 
Under Assumption~\ref{as:excitation_attack}, we have for a fixed $\|u\|_2= 1$ that 
\begin{align}\label{lemma3}
    \mathbb{P}\Bigr(  |u^T \phi(x_t)| \ge \frac{\lambda}{2}\, \Bigr|\,\mathcal{F}_{t-1}  \Bigr) \ge \frac{\bar c \lambda^4}{ (\sqrt{n} L \sigma)^4},
\end{align}
only for instances $t\in\mathcal{T}_{non}$, where $\bar c >0$ is a constant (see (25) in \cite{kim2026huber}). 
 Assumption~\ref{as:excitation} allows us to extend this lower bound to all time instances $t\in\mathcal{N}$. Noting that the expected cardinality of $\mathcal{T}_{non}$ is $p|\mathcal{N}|$,  a lower bound on $\inf_{\|u\|=1} \sum_{t\in \mathcal{N}} |u^T \phi(x_t)|$ is tightened accordingly. Following the proof in \cite{kim2026huber}, this ultimately improves the time bound and the positivity gap in Lemma~\ref{lem:excitation} by  a factor of $p^2$ and $1/p$, respectively. 
\hfill $\blacksquare$

\subsubsection{Proof of Theorem \ref{thm:stage1est}}
Note that a set of problems \eqref{stage1} can be expressed as 
    $\min_{a_i} f^{(i)}(a_i) := \sum_{t=0}^{T-1} |x_{t+1}^{(i)}-a_i^T \phi(x_t)| = \sum_{t=0}^{T-1} |(\bar a_i-a_i)^T \phi(x_t)+e_t^{(i)}+k_t^{(i)}|$.
Considering the definition of $\tilde f^{(i)} (a_i)$ in Lemma \ref{thm:noiseaware}, the triangle inequality ensures that for all $a_i\in \mathbb{R}^n$, it holds that
        $|f^{(i)}(a_i)-\tilde f^{(i)}(a_i)|\le \sum_{t=0}^{T-1}|e_t^{(i)}|$. This is followed by 
    \begin{align*}\label{upbo}
      \nonumber   \tilde f^{(i)}(\mathring a_i)&\le  f^{(i)}(\mathring a_i)+\sum_{t=0}^{T-1}|e_t^{(i)}|\\&\le  f^{(i)}( \bar a_i)+\sum_{t=0}^{T-1}|e_t^{(i)}|\le  \tilde f^{(i)}( \bar a_i)+2\sum_{t=0}^{T-1}|e_t^{(i)}|,
    \end{align*}
    where the second inequality stems from the optimality of $\mathring a_i$ to $f^{(i)}$. Now, by Lemma \ref{thm:noiseaware}, we attain under \eqref{timebound} that 
    \begin{align}\label{leftright}
    c T  \frac{ (1-2p)\lambda^5  }{n^2 L^4 \sigma^4}  \|\mathring a_i-\bar a_i\|_2\le    \tilde f^{(i)}(\mathring a_i) -   \tilde f^{(i)}( \bar a_i) \le  2\sum_{t=0}^{T-1}|e_t^{(i)}|,
    \end{align}
   where $\sum_{t=0}^{T-1}|e_t^{(i)}|=\mathcal{O}(T\sigma_e)$ with probability at least $1-\exp(-\Omega(T))$. This holds because 
    the independence of $e_t^{(i)}$ and the centering lemma \cite[Lemma 2.7.8]{vershynin2025high} imply that $\sum_{t=0}^{T-1}|e_t^{(i)}|-\mathbb{E}[|e_t^{(i)}|]$ has a $\psi_2$-norm of $\mathcal{O}(\sqrt{T}\sigma_e)$. Applying Hoeffding's inequality, we obtain
   \begin{align*}
       \mathbb{P}\Bigr(\sum_{t=0}^{T-1}|e_t^{(i)}|-\mathbb{E}[|e_t^{(i)}|] \le T\sigma_e\Bigr)\ge 1-\exp\Bigr(-\Omega\Bigr(\frac{T^2 \sigma_e^2}{T\sigma_e^2}\Bigr)\Bigr).
   \end{align*}
  Combining this with  $\sum_{t=0}^{T-1}\mathbb{E}[\vert{}e_t^{(i)}\vert{}] = \mathcal{O}(T\sigma_e)$ establishes the bound. 
   Rearranging the leftmost and rightmost terms in \eqref{leftright} completes the proof. 
\hfill $\blacksquare$

\subsubsection{Proof of Lemma \ref{thm:detect}}

   From Theorem \ref{thm:stage1est}, we obtain a universal upper bound on $\|\mathring a_i-\bar a_i\|_2\le \tau \sigma_e, ~\forall i$ with probability at least $1-\delta$.
   Then, we obtain
\begin{align*}
    |e_t^{(i)}+k_t^{(i)}|&- \tau\sigma_e\|\phi(x_t)\|_2\le |x_{t+1}^{(i)} - \mathring a_i^T \phi(x_t)|\\& \le \beta_1 \|\phi(x_t)\|_2 + \beta_2 = \alpha_1\tau\sigma_e \|\phi(x_t)\|_2 + \alpha_2 \sigma_e
\end{align*}
for all $t\in \mathcal{T}_i$ and all $i\in\{1,\dots,n\}$,
where the first inequality follows from  $|(\bar a_i- \mathring a_i)^T \phi(x_t)| \le \| \bar a_i- \mathring a_i\|_2 \|\phi(x_t)\|_2$  and the second follows from the  threshold rule \eqref{filter}. This completes the proof.
\hfill $\blacksquare$

\subsection{Proofs in Section \ref{sec:stage2}}\label{proofvb3}

\subsubsection{Proof of Lemma \ref{xxupperbound}}

    Notice that the spectral norm is bounded by the trace, meaning $\|\mathbf{X}_i^T\mathbf{X}_i\|_2 \le \operatorname{tr}(\mathbf{X}_i^T\mathbf{X}_i) = \sum_{t\in\mathcal{T}_i} \|\phi(x_t)\|_2^2$. 
    Applying the triangle inequality to the $\psi_1$-norm\footnote{For a sub-Gaussian variable $x$ with $\psi_2$-norm $\sigma$, $x^2$ is a sub-exponential variable with $\psi_1$-norm $\|x^2\|_{\psi_1} = \|x\|_{\psi_2}^2=\sigma^2$. The notion of sub-exponential variables are introduced in Section 2.8, \cite{vershynin2025high}.} yields
    \begin{align*}
        \Bigr\| \sum_{t\in\mathcal{T}_i} \|\phi(x_t)\|_2^2 \Bigr\|_{\psi_1} &\le \sum_{t\in\mathcal{T}_i} \bigr\|\|\phi(x_t)\|_2^2 \bigr\|_{\psi_1}= |\mathcal{T}_i|  \Bigr(\frac{\sqrt{n} L\sigma}{1-\rho L}\Bigr)^2,
    \end{align*}
    where we use the $\psi_2$-norm of $\|\phi(x_t)\|_2$ given in Lemma \ref{phixtnorm}.
    We complete the proof by applying the definition of sub-exponential tail bounds.
\hfill $\blacksquare$

\subsubsection{Proof of Lemma \ref{xylowerbound}}

Let 
\begin{align}\label{minmax}
    \Gamma_\text{min} = \frac{\lambda^2 I}{2}|\mathcal{T}_i|, \quad \Gamma_\text{max} =\mathcal{O}\left(|\mathcal{T}_i|\bigr(\frac{\sqrt{n} L\sigma}{1-\rho L}\bigr)^2 \log\bigr(\frac{1}{\delta}\bigr) \right) I,
\end{align}
which are  lower and upper bounds of $\mathbf{X}_i^T \mathbf{X}_i\in\mathbb{R}^{m\times m}$ obtained from Lemmas \ref{xxlowerbound} and  \ref{xxupperbound}. 
Let the singular decomposition of $\mathbf{X}_i$ be $ U\Sigma V^T$. 
Then, we can apply the arguments in Section D.2 of \cite{simchowitz2018learning} that, under the event that $\Gamma_\text{min}\preceq \mathbf{X}_i^T \mathbf{X}_i\preceq \Gamma_\text{max}$, we have
\begin{align*}
   & \mathbb{P}\left(\{\|(\mathbf{w}_i-\mathbf{\bar w}_i)^T U\|_2 > K \}~\cap ~\{\Gamma_\text{min}\preceq \mathbf{X}_i^T \mathbf{X}_i\preceq \Gamma_\text{max}\}\right) \\&\hspace{10mm}
   \le 45^m \biggr(\det \Bigr(\frac{32}{\bar p^2} \Gamma_{\text{max}}\Gamma_{\text{min}}^{-1}\Bigr)\biggr)^2 \exp\left(-\frac{K^2}{96\sigma_w^2}\right),\numberthis\label{expexp}
\end{align*}
where  $\sigma_w$ is a sub-Gaussian parameter of each entry of $\mathbf{w}_i-\mathbf{\bar w}_i$ and thus $\mathcal{O}(\sigma)$ due to the centering lemma  \cite[Lemma 2.7.8]{vershynin2025high}. Notice that the constant $\bar p$ originates from the Block Martingale Small-Ball (BMSB)  condition \cite[Definition 2.1]{simchowitz2018learning}: A process satisfies the $(k,\Gamma_{sb}, \bar p)$-BMSB condition  if for every $\|u\|_2=1$, it holds that
\begin{align*}
    \frac{1}{k} \sum_{t=1}^k \mathbb{P}(|u^T \phi(x_t)|\ge \sqrt{u^T \Gamma_{sb}u} ~|~\mathcal{F}_{t-1}) \ge \bar p.
\end{align*}
 The inequality \eqref{lemma3} establishes the $\bigr(1, (\frac{\lambda}{2})^2 I, \Omega(\frac{\lambda^4}{(\sqrt{n} L\sigma)^4})\bigr)$-BMSB condition. The expression \eqref{expexp} then evaluates to
$ \mathcal{O}\Bigr( \left[ \frac{n^5 L^{10}\sigma^{10}}{\lambda^{10} (1-\rho L)^2} \log\left(\frac{1}{\delta}\right) \right]^{2m} \exp\bigr(-\frac{K^2}{\sigma^2}\bigr) \Bigr)$,
which is upper bounded by $\delta$ when
\begin{align}\label{xxthird}
   K = \Omega\Biggr( \sigma \sqrt{  m \log\Bigr( \frac{n L \sigma}{\lambda (1-\rho L)\delta}\Bigr) } \Biggr).
\end{align}
Let the events $\{\Gamma_\text{min} \preceq \mathbf{X}_i^T \mathbf{X}_i\}$, $\{\mathbf{X}_i^T \mathbf{X}_i\preceq \Gamma_\text{max} \}$ and $(\{\|(\mathbf{w}_i-\mathbf{\bar w}_i)^T U\|_2 > K \}~\cap ~\{\Gamma_\text{min}\preceq \mathbf{X}_i^T \mathbf{X}_i\preceq \Gamma_\text{max}\})^{c}$ occur each with probability at least $1-\frac{\delta}{3}$. We then have
\begin{align*}
   & \|(\mathbf{w}_i-\mathbf{\bar w}_i)^T U\|_2  \le K\quad  \text{and} \quad \frac{\lambda^2I}{2}|\mathcal{T}_i| \preceq \mathbf{X}_i^T \mathbf{X}_i
\end{align*}
concurrently holding with probability at least $1-\delta$. This does not affect the order of time \eqref{finaltimebound} or constants \eqref{minmax} and \eqref{xxthird}. Noting that the smallest singular value of $\mathbf{X}_i$ is lower-bounded by $\sqrt{\lambda^2|\mathcal{T}_i|/2}$, we arrive at
\begin{align*}
   & \|(\mathbf{w}_i-\mathbf{\bar w}_i)^T \mathbf{X}_i(\mathbf{X}_i^T \mathbf{X}_i)^{-1}\|_2 \le \frac{\|(\mathbf{w}_i-\mathbf{\bar w}_i)^TU\|_2 }{\sqrt{\lambda^2|\mathcal{T}_i|/2}} \\&\hspace{22mm}=\mathcal{O}\left(\frac{\sigma}{\lambda \sqrt{|\mathcal{T}_i|}}\cdot \sqrt{ m\log\left( \frac{n L \sigma}{\lambda (1-\rho L)\delta} \right) } \right),
\end{align*}
with probability at least $1-\delta$. This completes the proof. 
\hfill $\blacksquare$

\subsubsection{Proof of Lemma \ref{asdf}}
Prior to proving this lemma, we first establish the following supporting result.
\begin{lemma}\label{wsubg}
     Consider a zero-mean sub-Gaussian variable $e$ with $\|e\|_{\psi_2} = \sigma_e$. Then, for a constant $\alpha>0$, we have  
     \begin{align*}
         &\mathbb{E}[|e|\cdot \mathbb{I}\{|e|> \alpha \sigma_e\}] \le \sigma_e \Bigr(\alpha+\frac{1}{2 \alpha}\Bigr)\exp\left(1- \alpha^2\right).
     \end{align*}
 \end{lemma}

\begin{proof}
    We apply integration by parts to the expectation of a non-negative random variable to evaluate $\mathbb{E}[\vert{}e\vert{}\cdot \mathbb{I}\{\vert{}e\vert{}>s\}]$ for $s\ge 0$, deriving
\begin{align*}
\mathbb{E}[|e| \cdot \mathbb{I}\{|e| > s\}] &= s \mathbb{P}(|e| > s) + \int_{s}^\infty \mathbb{P}(|e| > x) dx.
\end{align*}
By the definition of the $\psi_2$-norm $\Vert{}e\Vert{}_{\psi_2} = \sigma_e$, the two-sided tail probability is bounded by 
\begin{equation}\label{markov}
\begin{split}
    \mathbb{P}(\vert{}e\vert{} > x)
    &\le \frac{\mathbb{E}[\exp(e^2/\sigma_e^2)]}{\exp(x^2 /\sigma_e^2)}\le \exp\Bigr(1 - \frac{x^2}{\sigma_e^2}\Bigr)
    \end{split}
\end{equation}
for all $x \ge 0$, where the first inequality comes from Markov's inequality. We directly obtain
\begin{align*}
s\mathbb{P}(|e|>s) \le s\exp\left(1 - \frac{s^2}{\sigma_e^2}\right).
\end{align*}
For the remaining term, we have
\begin{align*}
\int_s^\infty \mathbb{P}(|e| > x) dx &\le \int_s^\infty \exp\Bigr(1 - \frac{x^2}{\sigma_e^2}\Bigr) dx \le \frac{\sigma_e^2}{2s} \exp\Bigr(1 - \frac{s^2}{\sigma_e^2}\Bigr)
\end{align*}
due to the standard upper bound for the Gaussian tail integral stating that $\int_x^\infty \exp(-au^2) du \le \frac{1}{2ax} \exp(-ax^2)$ for $a, x > 0$. Substituting $s = \alpha\sigma_e$ yields
\begin{align*}
\mathbb{E}[|e| \cdot \mathbb{I}{|e| > \alpha\sigma_e}]  &\le \alpha \sigma_e\exp\left(1- \alpha^2\right) +\frac{\sigma_e}{2 \alpha} \exp\left(1- \alpha^2\right),
\end{align*}
which completes the proof.
\end{proof}
We now present the proof of Lemma \ref{asdf}.
     Recall that $\mathbf{\bar w}_i$ is the concatenation of the sequence $\{\mathbb{E}[e_t^{(i)}+k_t^{(i)}\;|\;\mathcal{F}_t]\}_{t\in \mathcal{T}_i}$. We separate this sequence to attacked and clean data; i.e.,
 \begin{align*}
   \Gamma_{i} = \{t\in \mathcal{T}_i: k_t^{(i)}\ne 0\}, \quad  \tilde \Gamma_{i}= \{t\in \mathcal{T}_i: k_t^{(i)}= 0\}. 
 \end{align*}
Accordingly define $\mathbf{\bar w}_{i,1}$ as the concatenation of the sequence $\{\mathbb{E}[e_t^{(i)}+k_t^{(i)}\;|\;\mathcal{F}_t]\}_{t\in \Gamma_{i}}$ and $\mathbf{\bar w}_{i,2}$ as that of $\{\mathbb{E}[e_t^{(i)}+k_t^{(i)}\;|\;\mathcal{F}_t]\}_{t\in \tilde \Gamma_{i}}$. Define $\mathbf{X}_{i,1}$ and $\mathbf{X}_{i,2}$ in a similar fashion.  
 Since $\mathbf{\bar w}_i^T \mathbf{X}_i=\mathbf{\bar w}_{i,1}^T \mathbf{X}_{i,1}+\mathbf{\bar w}_{i,2}^T \mathbf{X}_{i,2}$, we separate the term as 
 \begin{align*}
     \|\mathbf{\bar w}_i^T \mathbf{X}_i\|_2 \le \|\mathbf{\bar w}_{i,1}^T \mathbf{X}_{i,1}\|_2+\|\mathbf{\bar w}_{i,2}^T \mathbf{X}_{i,2}\|_2.
 \end{align*}
We bound each term. First, from Lemma \ref{thm:detect}, we have
\begin{align*}
    |e_t^{(i)}+k_t^{(i)}| \le [(1+\alpha_1)\tau\|\phi(x_t)\|_2 + \alpha_2]\cdot \sigma_e 
\end{align*}
for all $t\in \Gamma_{i}$, with probability at least $1-\delta$, given that $T$ satisfies \eqref{timebound}. This gives 
\begin{align*}
  & \| \mathbf{\bar w}_{i,1}^T \mathbf{X}_{i,1}\|_2=\left\|\sum_{t\in\Gamma_{i}} (e_t^{(i)}+k_t^{(i)})\phi(x_t)\right\|_2\\&\hspace{3mm}\le \sum_{t\in\Gamma_{i}} |e_t^{(i)}+k_t^{(i)}| \|\phi(x_t)\|_2 \\&\hspace{3mm}\le \sum_{t\in\Gamma_{i}} [(1+\alpha_1)\tau\|\phi(x_t)\|_2 + \alpha_2]\cdot \sigma_e \cdot \|\phi(x_t)\|_2\\&
  \hspace{3mm}=(1+\alpha_1)\tau\sigma_e\sum_{t\in \Gamma_{i}}  \|\phi(x_t)\|_2^2 + \alpha_2 \sigma_e \sum_{t\in \Gamma_{i}} \|\phi(x_t)\|_2.\numberthis\label{bound11}
\end{align*}

We now aim to bound each entry of $\mathbf{\bar w}_{i,2}$ associated with $\tilde \Gamma_i$, where $k_t^{(i)}=0.$ 
The thresholding rule \eqref{filter} may cause some of these clean data to be discarded, meaning that the corresponding $e_t^{(i)}$ may not necessarily maintain a zero mean. In particular, given that $k_t^{(i)} = 0$, 
\begin{align*}
    \mathbb{E}[e_t^{(i)} ] = \mathbb{P}(t\in \tilde \Gamma_i )\mathbb{E}[e_t^{(i)} \;|\;t\in\tilde \Gamma_i] + \mathbb{E}[e_t^{(i)} \cdot \mathbb{I}\{t\notin \tilde \Gamma_i \}] = 0,
\end{align*}
which rearranges to
\begin{align}\label{denomnum}
    |\mathbb{E}[e_t^{(i)} \;|\;t\in\tilde \Gamma_i] | = \frac{| \mathbb{E}[e_t^{(i)} \cdot \mathbb{I}\{t\notin \tilde \Gamma_i\}]|}{\mathbb{P}(t\in \tilde \Gamma_i )}\le \frac{ \mathbb{E}[|e_t^{(i)}|\cdot  \mathbb{I}\{t\notin \tilde \Gamma_i\}]}{\mathbb{P}(t\in \tilde \Gamma_i )}. 
\end{align}
Provided that $T$ satisfies \eqref{timebound}, we have
\begin{align*}
|x_{t+1}^{(i)} - \mathring a_i^T \phi(x_t)| &= |(\bar a_i- \mathring a_i)^T \phi(x_t) + e_t^{(i)}|\\& \le \tau \sigma_e \|\phi(x_t)\|_2 + |e_t^{(i)}|.
\end{align*}
Thus, any discarded samples based on the thresholding rule (where $\alpha_1 \ge 1$) must satisfy $|e_t^{(i)}| > \alpha_2 \sigma_e$. This implies that $ \mathbb{E}[|e_t^{(i)}| \cdot \mathbb{I}\{t\notin \tilde \Gamma_i\}]$ is bounded by $ \mathbb{E}[|e_t^{(i)}| \cdot \mathbb{I}\{|e_t^{(i)}| > \alpha_2 \sigma_e\}]$. Similarly, we also have $\mathbb{P}(t\in \tilde \Gamma_i) \ge \mathbb{P}(|e_t^{(i)}| \le  \alpha_2 \sigma_e)$. 
Applying Lemma \ref{wsubg} and inequality \eqref{markov} to each of these terms yields a further bound, which extends inequality \eqref{denomnum} into
\begin{align}\label{eff}
      |\mathbb{E}[e_t^{(i)} \;|\;t\in\tilde \Gamma_i] |&\le \frac{\sigma_e (\alpha_2 + \frac{1}{2\alpha_2}) \exp(1-\alpha_2^2)}{1-\exp(1-\alpha_2^2)},
\end{align}
which holds since $\alpha_2 > 1$. Using $e^x-1\ge x$, we also have $\frac{\exp(1-\alpha_2^2)}{1-\exp(1-\alpha_2^2)}=\frac{1}{\exp(\alpha_2^2-1)-1}\le \frac{1}{\alpha_2^2 - 1}$.
Then, we can bound $\|\mathbf{\bar w}_{i,2}^T \mathbf{X}_{i,2}\|_2$ by noting that  each entry of $\mathbf{\bar w}_{i,2}$ is bounded by the right-hand side of \eqref{eff}. Thus, we obtain 
\begin{align*}
    \|\mathbf{\bar w}_{i,2}^T \mathbf{X}_{i,2}\|_2\le \sigma_e \Bigr(\alpha_2 + \frac{1}{2\alpha_2}\Bigr)\frac{1}{\alpha_2^2 - 1}\sum_{t\in |\tilde \Gamma_i| } \|\phi(x_t)\|_2,
\end{align*}
where $\bigr(\alpha_2 + \frac{1}{2\alpha_2}\bigr) \frac{1}{\alpha_2^2 - 1} < \frac{3}{2(\alpha_2 - 1)}$.
We can also bound
\begin{align}\label{bound22}
    \sum_{t\in |\tilde \Gamma_i| } \|\phi(x_t)\|_2 \le \sum_{t\in |\mathcal{T}_i|}\|\phi(x_t)\|_2\le\mathcal{O}\Bigr(|\mathcal{T}_i| \frac{\sqrt{n} L\sigma}{1-\rho L}\log\Bigr(\frac{1}{\delta}\Bigr)\Bigr) 
\end{align}
holding with probability at least $1-\delta$. We apply the union bound to \eqref{bound11} and \eqref{bound22} to derive the conclusion. 
 
\hfill $\blacksquare$

\subsubsection{Proof of Theorem \ref{thm:twostagefinal}}
Define the finite constant $M=\frac{1}{1-\exp(1-\alpha_2^2)}>0$.
As studied in the proof of Lemma \ref{asdf}, the probability that a clean data point is preserved in the filtered set is at least $\frac{1}{M}$. Also, the probability that a random time instance is associated with clean data is at least $\frac{1}{2}$ due to the definition of attacks. Since the two events are independent, we have $\mathbb{E}[|\mathcal{T}_i|]\ge \frac{T}{2M}$. From Chernoff's bound, we have 
    \begin{align*}
        \mathbb{P}\Bigr(|\mathcal{T}_i|\ge \frac{T}{4M}\Bigr)\ge 1-\exp\Bigr(-\frac{T}{16M}\Bigr), 
    \end{align*}
    which implies that  $|\mathcal{T}_i|= \Omega(T)$ holds with probability at least $1-\delta$ when $T= \Omega\left(\log(\frac{1}{\delta})\right)$.  Then, we apply the union bound to  Lemmas \ref{xxlowerbound},   \ref{xylowerbound},   \ref{asdf}, and the event $\{|\mathcal{T}_i|= \Omega(T)\}$ to bound each term in 
    \begin{equation*}\label{threelines}
    \begin{split}
      &\|\mathbf{ w}_i^T \mathbf{X}_i(\mathbf{X}_i^T \mathbf{X}_i)^{-1}\|_2 \\&\le   \underbrace{\|(\mathbf{w}_i-\mathbf{\bar w}_i)^T\mathbf{X}_i (\mathbf{X}_i^T \mathbf{X}_i)^{-1}\|_2}_{\text{Lemma \ref{xylowerbound}}}  + \underbrace{\|\mathbf{\bar w}_i^T \mathbf{X}_i\|_2}_{\text{Lemma \ref{asdf}}} \underbrace{\|(\mathbf{X}_i^T \mathbf{X}_i)^{-1}\|_2}_{\text{Lemma \ref{xxlowerbound}}}
    \end{split}
    \end{equation*}
     with probability at least $1-\delta$ when $T$ satisfies \eqref{timebound} and \eqref{finaltimebound}, where \eqref{timebound} dominates \eqref{finaltimebound}.    
  To ensure that our bound holds  for all $i\in\{1,\dots, n\}$ via the union bound, it suffices to replace $\delta$ with $\frac{\delta}{n}$ in \eqref{timebound} and the aforementioned lemmas. Note that the replacement does not affect the order of time \eqref{timebound}. This completes the proof.
\hfill $\blacksquare$

\renewcommand*{\bibfont}{\small}
\printbibliography

@book{Ortega2000,
  author    = {Ortega, James M. and Rheinboldt, Werner C.},
  title     = {Iterative Solution of Nonlinear Equations in Several Variables},
  publisher = {Society for Industrial and Applied Mathematics},
  year      = {2000}
}

@book{ledoux1991probability,
  title     = {Probability in Banach Spaces: Isoperimetry and Processes},
  author    = {Ledoux, Michel and Talagrand, Michel},
  year      = {1991},
  publisher = {Springer-Verlag},
  address   = {Berlin, Heidelberg}
}

@article{kim2026huber,
  title={Huber-based Robust System Identification with Near-Optimal Guarantees Across Independent and Adversarial Regimes},
  author={Jihun Kim and Javad Lavaei},
   journal={arXiv preprint arXiv:2603.27586},
  note={to appear in \textit{IEEE Conference on Decision and
Control}, 2026.}
}

@article{zhang2025exact,
  title={Exact Recovery Guarantees for Parameterized Nonlinear System Identification Problem under Sparse Disturbances or Semi-Oblivious Attacks},
  author={Haixiang Zhang and Baturalp Yalcin and Javad Lavaei and Eduardo D. Sontag},
journal={Transactions on Machine Learning Research},
  year={2025},
 publisher={JMLR}
}

@inproceedings{kim2025prevailing,
  title={Prevailing against Adversarial Noncentral Disturbances: Exact Recovery of Linear Systems with the $\ell_1$-norm Estimator},
  author={Kim, Jihun and Lavaei, Javad},
  booktitle={American Control Conference},
  pages={1161--1168},
  year={2025}
}

@article{Narendra1987,
  author  = {Narendra, Kumpati S. and Annaswamy, Anuradha M.},
  title   = {A new adaptive law for robust adaptation without persistent excitation},
  journal = {International Journal of Control},
  volume  = {45},
  number  = {1},
  pages   = {127--160},
  year    = {1987},
  publisher = {Taylor \& Francis}
}

@inproceedings{kim2026sharp,
  title={On the Sharp Input-Output Analysis of Nonlinear Systems under Adversarial Attacks},
  author={Jihun Kim and Yuchen Fang and Javad Lavaei},
  booktitle={International Conference on Machine Learning},
  pages={},
  year={2026}
}

@article{kumar2025machine,
  title={Machine learning in parameter estimation of nonlinear systems},
  author={Kumar, Kaushal and Kostina, Ekaterina},
  journal={The European Physical Journal B},
  volume={98},
  year={2025},
  publisher={Springer}
}

@book{vasilescu2005electronic,
  author    = {Vasilescu, Gabriel},
  title     = {Electronic Noise and Interfering Signals: Principles and Applications},
  date      = {2005},
  publisher = {Springer},
  location  = {Berlin, Heidelberg}
}

@article{ghodeswar2025quantifying,
  author       = {Ghodeswar, Archana and Bhandari, Mahabir and Hedman, Bruce},
  title        = {Quantifying the economic costs of power outages owing to extreme events: A systematic review},
  journaltitle = {Renewable and Sustainable Energy Reviews},
  volume       = {207},
  date         = {2025}
}

@book{freidlin2012random,
  author    = {Freidlin, Mark I. and Wentzell, Alexander D.},
  title     = {Random Perturbations of Dynamical Systems},
  series    = {Grundlehren der mathematischen Wissenschaften},
  volume    = {260},
  edition   = {3},
  publisher = {Springer},
  location  = {New York, NY},
  date      = {2012}
}

@article{teixeira2015secure,
  author       = {Teixeira, Andr{\'e} and Shames, Iman and Sandberg, Henrik and Johansson, Karl Henrik},
  title        = {A secure control framework for resource-limited adversaries},
  journaltitle = {Automatica},
  volume       = {51},
  pages        = {135--148},
  date         = {2015},
  publisher    = {Elsevier}
}

@article{allam2021analyzing,
  author       = {Allam, Ahmed and Feuerriegel, Stefan and Rebhan, Michael and Krauthammer, Michael},
  title        = {Analyzing Patient Trajectories With Artificial Intelligence},
  journaltitle = {Journal of Medical Internet Research},
  volume       = {23},
  number       = {12},
  pages        = {e29812},
  date         = {2021}
}

@article{huber1964robust,
  title={Robust Estimation of a Location Parameter},
  author={Huber, Peter J.},
  journal={The Annals of Mathematical Statistics},
  volume={35},
  number={1},
  pages={73--101},
  year={1964},
  publisher={Institute of Mathematical Statistics}
}

@book{ljung1998sys,
  author    = {Ljung, Lennart},
  title     = {System Identification},
  subtitle  = {Theory for the User},
  edition   = {2},
  date      = {1998},
  publisher = {Prentice Hall},
  location  = {Upper Saddle River, NJ}
}

@book{koopmans1950statistical,
  editor    = {Koopmans, Tjalling C.},
  title     = {Statistical Inference in Dynamic Economic Models},
  series    = {Cowles Commission Monograph},
  number    = {10},
  publisher = {John Wiley \& Sons},
  location  = {New York},
  year      = {1950}
}

@inproceedings{sarkar2019near,
  title={Near optimal finite time identification of arbitrary linear dynamical systems},
  author={Tuhin Sarkar and Alexander Rakhlin},
  booktitle={International Conference on Machine Learning},
  pages={5610--5618},
  year={2019}
}

@inproceedings{simchowitz2019semi,
  title={Learning Linear Dynamical Systems with Semi-Parametric Least Squares},
  author={Max Simchowitz and Ross Boczar and Benjamin Recht},
  booktitle={Conference on Learning Theory},
  volume={99},
  pages={1--89},
  year={2019}
}

@book{vershynin2025high,
  title={High-Dimensional Probability: An Introduction with Applications in Data Science},
  author={Roman Vershynin},
  year={2026},
    edition={2},
  publisher={Cambridge University Press},
  address   = {Cambridge}
}

@book{widrow2008quantization,
  title={Quantization Noise: Roundoff Error in Digital Computation, Signal Processing, Control, and Communications},
  author={Widrow, Bernard and Koll{\'a}r, Istv{\'a}n},
  year={2008},
  publisher={Cambridge University Press},
  address={Cambridge}
}

@article{brunton2016discovering,
  title={Discovering governing equations from data by sparse identification of nonlinear dynamical systems},
  author={Brunton, Steven L and Proctor, Joshua L and Kutz, J Nathan},
  journal={Proceedings of the National Academy of Sciences},
  volume={113},
  number={15},
  pages={3932--3937},
  year={2016},
  publisher={National Acad Sciences}
}

@article{zhang2015optimal,
  title={Optimal denial-of-service attack scheduling with energy constraint},
  author={Zhang, Heng and Cheng, Peng and Shi, Ling and Chen, Jiming},
  journal={IEEE Transactions on Automatic Control},
  volume={60},
  number={11},
  pages={3023--3028},
  year={2015},
  publisher={IEEE}
}

@book{bentley2005principles,
  title={Principles of Measurement Systems},
  author={Bentley, John P.},
  year={2005},
  publisher={Pearson Education},
  edition={4}
}

@incollection{vanderziel1978noise,
  title={Noise in Solid State Devices},
  author={Van Der Ziel, Aldert and Chenette, Eugene R.},
  booktitle={Advances in Electronics and Electron Physics},
  volume={46},
  pages={313--383},
  year={1978},
  publisher={Academic Press}
}

@article{fawzi2014secure,
  title={Secure Estimation and Control for Cyber-Physical Systems Under Adversarial Attacks},
  author={Hamza Fawzi and Paulo Tabuada and Suhas Diggavi},
  journal={IEEE Transactions on Automatic Control},
  volume={59},
  number={6},
  pages={1454--1467},
  year={2014},
  publisher={IEEE}
}

@article{pajic2017attack,
  title={Attack-Resilient State Estimation for Noisy Dynamical Systems},
  author={Miroslav Pajic and Insup Lee and George J. Pappas},
  journal={IEEE Transactions on Control of Network Systems},
  volume={4},
  number={1},
  pages={82--92},
  year={2017},
  publisher={IEEE}
}

@article{pasqualetti2013cyber,
  title={Attack Detection and Identification in Cyber-Physical Systems},
  author={Fabio Pasqualetti and Florian D\"{o}rfler and Francesco Bullo},
  journal={IEEE Transactions on Automatic Control},
  volume={58},
  number={11},
  pages={2715--2729},
  year={2013},
  publisher={IEEE}
}

@inproceedings{kanakeri2025outlier,
  title={Outlier-Robust Linear System Identification Under Heavy-Tailed Noise},
  author={Vinay Kanakeri and Aritra Mitra},
  booktitle={Learning for Dynamics and Control Conference},
  pages={540--551},
  year={2025},
    volume = 	 {283},
  publisher =    {PMLR}
}

@article{SheOwen2011,
  author  = {She, Yiyuan and Owen, Art B.},
  title   = {Outlier Detection Using Nonconvex Penalized Regression},
  journal = {Journal of the American Statistical Association},
  year    = {2011},
  volume  = {106},
  pages   = {626--639}
}

@article{Gannaz2007,
  author  = {Gannaz, Irene},
  title   = {Robust estimation and wavelet thresholding in partially linear models},
  journal = {Statistics and Computing},
  year    = {2007},
  volume  = {17},
  pages={293--310}
}

@article{yalcin2024exact,
  title={Exact recovery for system identification with more corrupt data than clean data},
  author={Yalcin, Baturalp and Zhang, Haixiang and Lavaei, Javad and Arcak, Murat},
  journal={IEEE Open Journal of Control Systems},
  year={2024},
  publisher={IEEE}
}

@article{lugosi2021robust,
  title={Robust multivariate mean estimation: The optimality of
trimmed mean},
  author={Gabor Lugosi and Shahar Mendelson},
  journal={The Annals of Statistics},
  volume={49},
  number={1},
  pages={393--410},
  year={2021},
  publisher={Institute of Mathematical Statistics}
}

@book{hall1980martingale,
  author    = {Hall, Peter and Heyde, Christopher C.},
  title     = {Martingale Limit Theory and Its Application},
  year      = {1980},
  publisher = {Academic Press},
  address   = {New York}
}

@inproceedings{simchowitz2018learning,
  title={Learning without mixing: Towards a sharp analysis of linear system identification},
  author={Simchowitz, Max and Mania, Horia and Tu, Stephen and Jordan, Michael I. and Recht, Benjamin},
  booktitle={Conference on Learning Theory},
  pages={439--473},
  year={2018}
}


\begin{IEEEbiography}[{\includegraphics[width=1in,height=1.25in,clip,keepaspectratio]{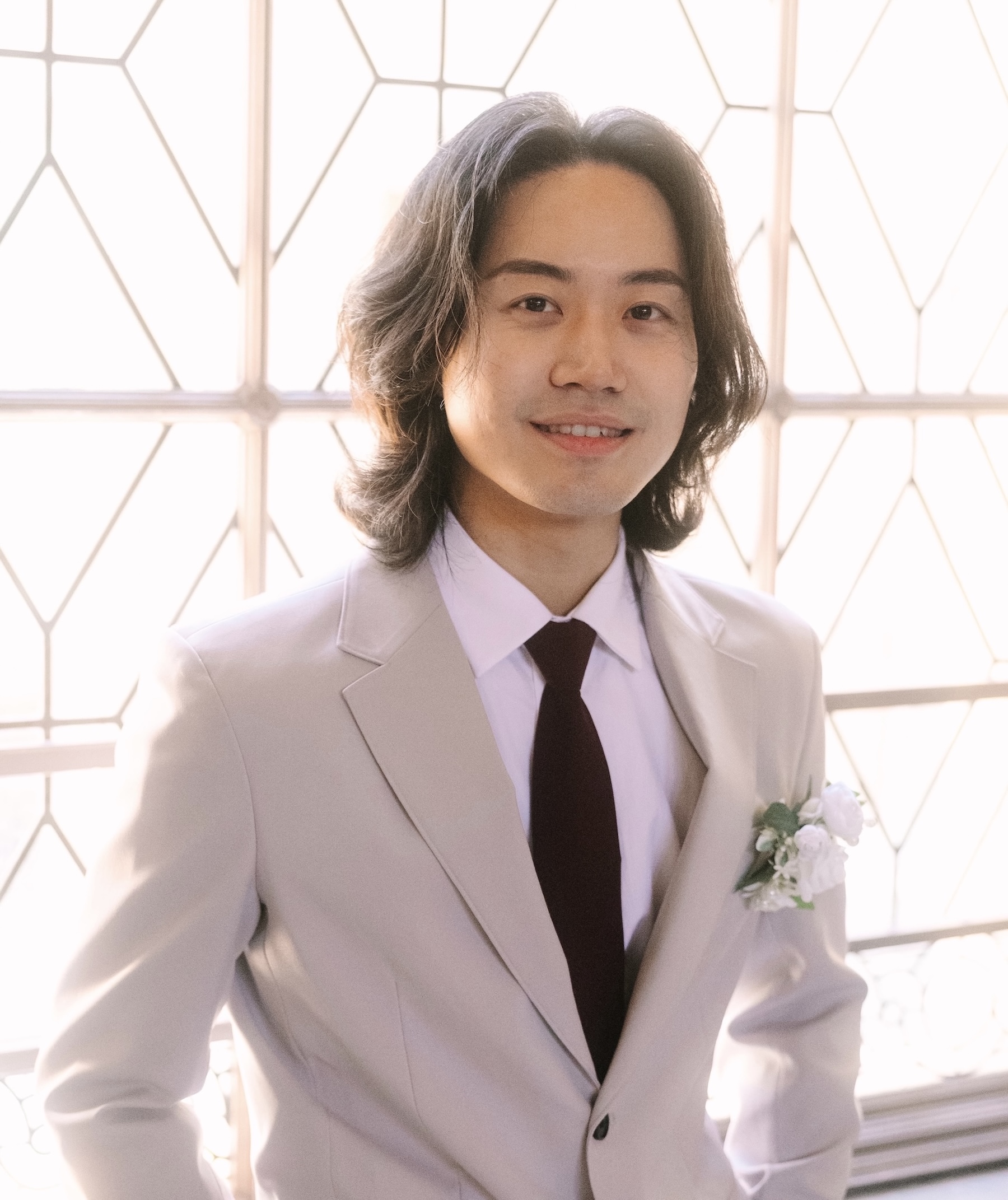}}]{Jihun Kim} (Graduate Student Member, IEEE) is a Ph.D. Candidate in Industrial Engineering and Operations Research at University of California, Berkeley, CA, USA. He obtained the B.S. degree in both Industrial Engineering and Statistics from Seoul National University in 2022. His research interests include theoretical foundations in optimization, control, machine learning, and system identification, particularly in the context of dynamical systems under attack and their applications to safety-critical systems.
\end{IEEEbiography}

\begin{IEEEbiography}[{\includegraphics[width=1in,height=1.25in,clip,keepaspectratio]{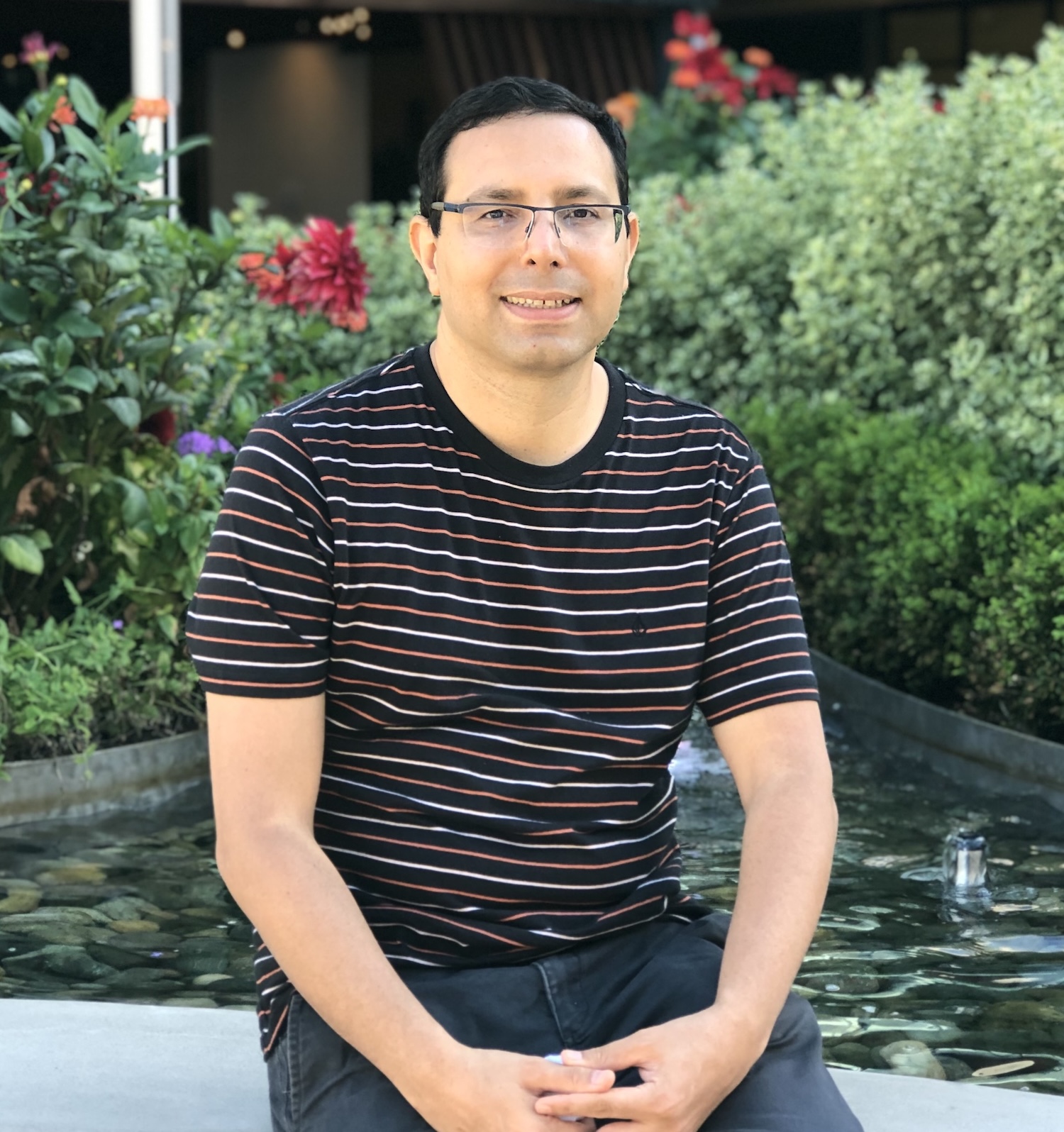}}]{Javad Lavaei} (Fellow, IEEE) is an Associate Professor in the Department of Industrial Engineering and Operations Research at UC Berkeley. He obtained the Ph.D. degree in Control \& Dynamical Systems from California Institute of Technology. He is a senior editor of the IEEE Systems Journal and has served on the editorial boards of the IEEE Transactions on Automatic Control, IEEE Transactions on Control of Network Systems, IEEE Transactions on Smart Grid, and IEEE Control Systems Letters. 
\end{IEEEbiography}

\end{document}